\documentclass[reqno]{amsart}

\usepackage{amsmath}
\usepackage{amssymb}
\usepackage{amsthm}
\usepackage[mathscr]{eucal}

\numberwithin{equation}{section}

\newtheorem{theorem}{Theorem}[section]
\newtheorem{lemma}[theorem]{Lemma}
\newtheorem{remark}[theorem]{Remark}
\newtheorem{proposition}[theorem]{Proposition}

\newtheorem{corollary}[theorem]{Corollary}
\newtheorem*{claim*}{Claim}

\begin{document}

\title[the uniqueness and the non-degeneracy]
{A note on the uniqueness and the non-degeneracy of positive radial solutions for semilinear elliptic problems and its application}

\author[S.~Adachi, M.~Shibata, T.~Watanabe]
{Shinji Adachi, Masataka Shibata, Tatsuya Watanabe}

\address{Shinji Adachi \newline
Department of Mathematical and Systems Engineering, Faculty of Engineering, Shizuoka University,
3-5-1 Johoku, Naka-ku, Hamamatsu, 432-8561, Japan}
\email{adachi@shizuoka.ac.jp}

\address{Masataka Shibata \newline
Department of Mathematics, Tokyo Institute of Technology,
2-12-1 Oh-okayama, Meguro-ku, Tokyo 152-8551, Japan}
\email{shibata@math.titech.ac.jp}

\address{Tatsuya Watanabe \newline
Department of Mathematics, Faculty of Science, Kyoto Sangyo University,
Motoyama, Kamigamo, Kita-ku, Kyoto-City, 603-8555, Japan}
\email{tatsuw@cc.kyoto-su.ac.jp}

\begin{abstract}
In this paper, we are concerned with the uniqueness 
and the non-degeneracy of positive radial solutions 
for a class of semilinear elliptic equations.
Using detailed ODE analysis, we extend previous results to
cases where nonlinear terms may have sublinear growth.
As an application, we obtain the uniqueness and the non-degeneracy of
ground states for modified Schr\"odinger equations.
\end{abstract}


\maketitle

\section{Introduction and Main results}

In this paper, we consider the uniqueness and the non-degeneracy
of positive radial solutions for the following semilinear elliptic problem:
\begin{equation}\label{eq:1.1}
-\Delta u= g(u) \ \hbox{in} \ \mathbb{R}^N, \ 
u(x) \to 0 \ \hbox{as} \ |x| \to \infty,
\end{equation}
where $N \ge 2$ and $g:[0,\infty) \to \mathbb{R}$.
Especially we are interested in the case where the nonlinear term $g$ may have
a sublinear growth at infinity.

Equation \eqref{eq:1.1} appears in various fields of mathematical physics and biology, 
and has been studied widely.
Particularly, almost optimal conditions for the existence 
of nontrivial solutions have been obtained in 
\cite{BGK, BL, HIT}.
Equation \eqref{eq:1.1} has the variational structure
and solutions of \eqref{eq:1.1} can be characterized as critical points
of the functional $E:H^1(\mathbb{R}^N) \to \mathbb{R}$ defined by
$$
E(u):= \frac{1}{2} \int_{\mathbb{R}^N} | \nabla u|^2 \,dx 
-\int_{\mathbb{R}^N} G(u) \,dx,$$
where $G(s)=\int_0^s g(\tau) \,d\tau$.
Then nontrivial solutions obtained in \cite{BGK, BL, HIT}
were indeed {\it ground states}
of \eqref{eq:1.1}, namely solutions of \eqref{eq:1.1} which have least energy
among all nontrivial critical points of $E$.
It should be also noted that ground states obtained 
in \cite{BGK, BL, HIT} were positive and radially symmetric.
But under generic assumptions for 
which the existence of ground states is guaranteed, 
one can show that {\it any} ground state is positive.
Moreover we can also prove that any ground state is radially symmetric.
This can be shown in two ways: One is based on the Moving Plane method
for positive solutions (\cite{GNN}), and 
the other is based on the rearrangement technique for minimizers
(\cite{BJM} and references therein).
Thus if we could show the uniqueness of positive radial solutions,
then we can also obtain the uniqueness of ground states.

There are lots of works concerning with the uniqueness of positive radial solutions for \eqref{eq:1.1}.
See \cite{C, Kw} for the model case $g(u)=-u+u^3$ or $g(u)=-u+u^p$ 
and \cite{BS, Ko, Mc, MS, OS, PS1, PS2, ST} and references therein for more general nonlinearities. 
But there still remains a gap between sufficient conditions 
for the existence and those for the uniqueness.
Especially when the nonlinear term may have sublinear growth at infinity,
the uniqueness is not known in general.

Once we could obtain the uniqueness of positive solutions of \eqref{eq:1.1} and
go back to its proof, we can also show that the unique positive solution of \eqref{eq:1.1} is {\it non-degenerate}, 
that is, the linearized operator $\mathcal{L}
=\Delta +g'(u)$ has exactly $N$-dimensional kernel 
which is given by ${\rm span} \, \{ \frac{\partial u}{\partial x_i} \}$
(See also Corollary \ref{cor:2.4} below).
This type of non-degeneracy can be applied in various ways.
See for example \cite{FW, NT} for singularly perturbed problems,
\cite{M} for the existence of traveling waves of nonlinear Schr\"odinger equations 
and \cite{St} for the (in)stability of standing waves for nonlinear Schr\"odinger equations.

Our aim of this paper is to prove the uniqueness and the non-degeneracy
of positive radial solutions of \eqref{eq:1.1} for general nonlinearities to which
previous results cannot be applied.
Especially as we have mentioned above, we are interested in the case
where the nonlinear term may have sublinear growth at infinity.

To state our main result, 
we impose the following conditions on the nonlinear term $g$:
\begin{itemize}
\item[\rm(G1)] $g \in C^1[0,\infty)$, $g(0)=0$ and $g'(0)<0$.
\item[\rm(G2)] There exist $b>0$ 
and $\tilde{b} \in (b,\infty]$ 
such that
$g(s)<0$ on $(0,b)$, 
$g(s)>0$ on $(b, \tilde{b})$ and $g'(b)>0$.
If $\tilde{b}<\infty$, then $g(s)<0$ on $(\tilde{b},\infty)$ and $g'(\tilde{b})<0$.
\item[\rm(G3)] There exists $\zeta>b$ such that $\displaystyle G(\zeta)=\int_0^\zeta g(s)\,ds>0$.
\item[\rm(G4)] $K_g(s)$ is decreasing on $(b, \tilde{b})$.
\item[\rm(G5)] $K_g(s) \leq 1$  on $(0, b)$.
\end{itemize}
Here $K_g(s)$ is the {\it growth function} of $g(s)$ which is defined by
\begin{equation}\label{eq:1.2}
 K_g(s) := \frac{s g'(s)}{g(s)}.
\end{equation}
Roughly speaking, the function $K_g$ measures growth orders 
at zero and infinity. For example if $g(s)=-s+s^p$, we can easily see that 
$K_g(0)=1$ and $K_g(\infty)=p$. Especially when $K_g(\infty) \in (0,1)$,
the nonlinear term $g$ has sublinear growth at infinity.
We also note that if $\tilde{b}<\infty$, 
then it follows that $\| u\|_{L^{\infty}} < \tilde{b}$ 
by the Maximum Principle.

In this setting, we obtain the following result.

\begin{theorem}\label{thm:1.1}
Assume (G1)--(G5). Then \eqref{eq:1.1} has at most
one positive radial solution. Moreover 
the corresponding linearized operator $\mathcal{L}=\Delta +g'(u)$ 
does not have 0-eigenvalue in $H^1_{rad}(\mathbb{R}^N)$.
\end{theorem}

As for the existence of a positive radial solution of \eqref{eq:1.1},
it suffices to assume (G1), (G3) and the following additional conditions (\cite{BGK, BL, HIT}):
\begin{itemize}
\item[\rm(G6)] 
$\displaystyle \limsup_{s \to \infty} \frac{g(s)}{
s^{N+2 \over N-2}} \le 0$ if $N \ge 3$.

$\displaystyle \limsup_{s\to\infty}\frac{g(s)}{e^{\alpha s^2}}\leq 0$ for 
any $\alpha>0$ if $N=2$.
\end{itemize}
Although the existence of a positive radial solution can be obtained
under weaker assumptions, we don't give precise statements here.
By Theorem \ref{thm:1.1} and Corollary \ref{cor:2.4} below, we have the following result.

\begin{corollary}\label{cor:1.2}
Assume (G1)--(G6). 
Then the positive radial solution of \eqref{eq:1.1} is unique and non-degenerate.
\end{corollary}

Now since we are interested in the positive radial solutions, \eqref{eq:1.1}
reduces to the following ordinary differential equation:
\begin{equation}\label{eq:1.3}
\left\{
\begin{array}{l}
\displaystyle u'' + \frac{N-1}{r} u' + g(u) = 0 \text{ for } r>0, \\
\noalign{\smallskip}
 u'(0)=0, \quad u(r) \rightarrow 0 \text{ as } r \rightarrow \infty, \\
\noalign{\smallskip}
 u(r) > 0 \text{ for } r \geq 0.
\end{array}
\right.
\end{equation}
To obtain the uniqueness of positive radial solutions of \eqref{eq:1.3},
we apply the shooting method as in \cite{Kw, Mc}. More precisely,
let us consider the initial value $u(0)=d$ for $d>0$ and denote by
$R(d)$ the first zero of $u=u(r,d)$.
Then it suffices to show that there exists a unique $d_0>0$ such that
$R(d_0)=\infty$. To this end, we consider the variation
$\delta(r,d)= \frac{ \partial u}{\partial d}(r,d)$.
As in \cite{Kw, Mc}, 
the key is to construct a suitable comparison function which has the same zero as $\delta$.
Although a general method based on {\it I-theorem} has been established
in \cite{Mc}, this method may fail if the nonlinear term has a sublinear growth at infinity.
Thus our main purpose is the construction of the comparison function without using I-theorem.
(See also Remarks \ref{rem:2.6}, \ref{rem:3.7} below.)

As an application of Theorem \ref{thm:1.1}, we study 
the uniqueness and the non-degeneracy of ground states for the following 
{\it modified} Schr\"odinger equation:
\begin{equation}\label{eq:1.4}
-\Delta u+\lambda u-\kappa u \Delta (|u|^2)
=|u|^{p-1}u \quad \hbox{in} \ \mathbb{R}^N,
\end{equation}
where $N \ge 2$, $\lambda>0$, $\kappa>0$ and $p>1$. 
Then two advances on previous results can be made: 
One is the uniqueness for $N=2$ which was not studied in \cite{ASW1},
and the other is the non-degeneracy for $1<p<3$
which was not covered in previous results \cite{ASW1, S}.
(See also Remark \ref{rem:4.10} below.)
The nonlinear term $|u|^{p-1}u$ in \eqref{eq:1.4} has still superlinear growth
even when $1<p<3$. But in this case, as we will see in Remark \ref{rem:4.4} below,
problem \eqref{eq:1.4} has a sublinear structure because of the quasilinear term $u\Delta(|u|^2)$.

This paper is organized as follows. In section 2,
we classify the set of initial values as in \cite{Kw, Mc}
and prove Theorem \ref{thm:1.1}, 
leaving the proof of the main proposition into Section 3.
Section 3 is the most technical part and we will construct a suitable
comparison function by using detailed ODE analysis.
Finally in Section 4, we study the uniqueness of positive radial solutions for 
a class of quasilinear elliptic problems.
Especially we will apply Theorem \ref{thm:1.1} to obtain 
a completely optimal result for the uniqueness and the non-degeneracy of 
complex valued ground states of modified Schr\"odinger equations in Subsection 4.4.

\section{Classification of the set of initial values and Preliminaries}

In this section, we give some preliminaries and prove
Theorem \ref{thm:1.1}. To this end, we study the structure of radial solutions of the following ODE:
\begin{equation}
\label{eq:2.1}
\left\{
\begin{array}{ll}
\displaystyle u''+{N-1 \over r}u'+g(u)=0, & r \in (0,\infty), \\
\noalign{\smallskip}
u(0)=d >0. & \quad 
\end{array}
\right.
\end{equation}
For $d>0$, we denote the unique solution of \eqref{eq:2.1} by $u(\cdot,d)$.
Then as in \cite{Kw} and \cite{Mc}, 
we can classify the sets of initial values as follows:
\begin{align*}
N&= \{ 0<d<\tilde{b} \ ; \ \hbox{there exists} \ R=R(d)\in(0,\infty) \ 
\hbox{such that} \ u(R,d)=0 \}. \\
G&= \{ 0<d<\tilde{b} \ ; \ u(r,d)>0 \ \hbox{for all} \ r>0 \ \hbox{and} \ 
\lim_{r\to \infty} u(r,d)=0 \}. \\
P&=\{ 0<d<\tilde{b} \ ; \ u(r,d)>0 \ \hbox{for all} \ r>0 \ \\
&\hspace{7.2em} \hbox{but} \ 
u(r,d) \ \hbox{does not converge to zero at infinity}
\}. 
\end{align*}
When $d \in N$, we may assume that $R(d)$ is a first zero of $u(r,d)$.
Then it follows that  $u'(R)<0$. 
Furthermore when $d\in G$, we set $R(d)=\infty$.
In this setting, we have the following result.

\begin{proposition}\label{prop:2.1} We obtain the following properties.
\begin{itemize}
\item[\rm(i)] The sets $N$, $G$ and $P$ are disjoint,
$N \cup P \cup G= (0,\tilde{b})$ and $(0,b] \subset P$.
Especially if $u$ is a solution of \eqref{eq:1.3}, then $u(0)>b$.
Furthermore $N$ and $P$ are both open in $(0,\tilde{b})$.
\item[\rm(ii)] If $d\in N \cup G$, then the corresponding solution
$u(r,d)$ of \eqref{eq:2.1} is monotone decreasing with respect to $r$.
\item[\rm(iii)] If $d\in G$, then 
$u^{(k)}(r,d), k=0,1,2$ 
decay exponentially at infinity. Moreover it follows that
\begin{equation}\label{eq:2.2}
\lim_{r \to \infty} \frac{u'(r,d)}{u(r,d)} = - \sqrt{-g'(0)} <0.
\end{equation}
\end{itemize}
\end{proposition}

Next in order to study the uniqueness and the non-degeneracy 
of the positive radial solution of \eqref{eq:2.1}, 
we consider the following linearized problem:
\begin{equation}
\label{eq:2.3}
\left\{
\begin{array}{l}
\displaystyle \delta'' + \frac{N-1}{r} \delta' + g'(u) \delta=0, \\
\noalign{\smallskip}
\delta(0)=1, \ \delta'(0)=0.
\end{array}
\right.
\end{equation}
For a solution $u=u(\cdot, d)$ of \eqref{eq:2.1}, 
we denote the unique solution of \eqref{eq:2.3} by $\delta(\cdot,d)$ 
and sometimes we denote it by $\delta(\cdot)$ for simplicity.
Our first purpose is to establish the following result.

\begin{proposition}\label{prop:2.2}
Suppose (G1)-(G4) hold.
\begin{itemize}
\item[\rm(i)] If $d \in G$, then $\delta(r,d)$ has
at least one zero in $(0, \infty)$.
\item[\rm(ii)] Let $d_0= \inf( N \cup G)$.
Then $d_0 \in G$ and $\delta(r, d_0)$ has exactly one zero $r_{\delta}\in (0,\infty)$.
Moreover it follows that
\begin{equation*}
\delta(r)
\left\{
\begin{array}{cl}
>0 & \text{ if } \ 0 < r < r_\delta, \\
=0 & \text{ if } \ r=r_\delta, \\
<0 & \text{ if } \ r_\delta < r <\infty.
\end{array}
\right.
\end{equation*}
\item[\rm(iii)] Let $d\in G$ and suppose that $\delta(r) \to -\infty$
as $r \to \infty$. Then there exists $\varepsilon>0$ such that
$(d,d+\varepsilon) \subset N$ and $(d-\varepsilon,d) \subset P$.
\end{itemize}
\end{proposition}

\begin{proof}
When $\tilde{b}<\infty$, the proof can be found in \cite{OS}.
Thus it suffices to consider the case $\tilde{b}=\infty$.

To prove (i), 
we suppose by contradiction that $\delta(r,d)>0$ on $(0, \infty)$.
First differentiating \eqref{eq:2.1}, we have
\begin{equation*}
u'''+ \frac{N-1}{r} u'' + \left(g'(u)- \frac{N-1}{r^2} \right)u'=0.
\end{equation*}
Then by a direct calculation, we obtain the following Picone identity:
\begin{equation*}
\left[
r^{N-1} u' u'' - r^{N-1} (u')^2 \frac{\delta'}{\delta}
\right]'
=
\frac{N-1}{r} (u')^2 + r^{N-1}\left( u'' - u' \frac{\delta'}{\delta}\right)^2 
\ \hbox{on} \ (0,\infty).
\end{equation*}
Integrating it over $(r, \tilde{r})$
for any $(r,\tilde{r}) \subset (0,\infty)$, we get
\begin{align*}
& \int_r^{\tilde{r}} \frac{N-1}{s} (u')^2 \,ds \\
&\leq
\tilde{r}^{N-1} \left(u'(\tilde{r}) u''(\tilde{r}) - (u')^2(\tilde{r}) \frac{\delta'(\tilde{r})}{\delta(\tilde{r})} \right)
- r^{N-1} \left( u'(r) u''(r) - (u')^2(r) \frac{\delta'(r)}{\delta(r)} \right).
\end{align*}

Now since both $u$ and $u'$ decay exponentially at infinity, 
$u''=-\frac{N-1}{r} u' - g(u)$ also decays exponentially.
Moreover it follows that either $\delta(\tilde{r}), \delta'(\tilde{r}) \to 0$ 
or $\delta(\tilde{r}), \delta'(\tilde{r}) \to \infty$ as $\tilde{r} \to \infty$.
(See \cite{Mc} Lemma 2 (b) for the proof.)
In the former case, $\delta$, $\delta'$ decay exponentially 
and $\delta'/\delta \to -\sqrt{-g'(0)}$ as $\tilde{r} \to \infty$.
In the latter case, we also have 
$\delta'(\tilde{r})/\delta(\tilde{r})>0$ for large $\tilde{r}$.
Thus in both cases, we obtain 
\begin{equation*}
\limsup_{\tilde{r} \to \infty}
\tilde{r}^{N-1} \left(u'(\tilde{r}) u''(\tilde{r}) - (u')^2(\tilde{r}) \frac{\delta'(\tilde{r})}{\delta(\tilde{r})} \right) \leq 0.
\end{equation*}
On the other hand since 
$\delta(0)=1$, $\delta'(0)=0$ and $u'(0)=0$, we have from \eqref{eq:2.1} that
\begin{align*}
& \lim_{r \to 0}
r^{N-1} \left( u'(r) u''(r) - (u')^2(r) \frac{\delta'(r)}{\delta(r)} \right) \\
= &
\lim_{r \to 0} r^{N-1} u'(r) \left( - \frac{N-1}{r} u'(r) -  g(u) \right) =0.
\end{align*}
Thus it follows that 
\begin{align*}
\int_0^\infty \frac{N-1}{s} (u')^2 \,ds \leq 0.
\end{align*} 
This is a contradiction and hence $\delta$ has at least one zero.

Finally proofs of (ii) and (iii) 
can be done by similar arguments in \cite{Mc} Lemmas 3 (b) and 10.
\end{proof}

We note that if $\delta(r,d)$ has exactly one zero in $(0,R)$ 
for $d \in N \cup G$, then $d$ is called {\it admissible}.
Moreover if $d$ is admissible and $\delta(R,d)<0$ for $d \in N$ or
$\delta(r) \to -\infty$ as $r \to \infty$ for $d\in G$, 
then $d$ is said to be {\it strictly admissible}.
(See \cite{Kw} and \cite{Mc}.)
Our goal is to prove the following result 
which shows the strict admissibility of any admissible $d \in N\cup G$.
We will give its proof in Section 3.

\begin{proposition}\label{prop:2.3}
Assume (G1)-(G5) and let $d\in N \cup G$.
Suppose further that $\delta(r,d)$ has exactly one zero in $(0,R)$.
Then it follows that $\delta(R,d)<0$ for $d \in N$ or
$\displaystyle \lim_{r \to \infty} \delta(r) =-\infty$ for $d\in G$.
\end{proposition}

Now we are ready to prove Theorem \ref{thm:1.1}.

\begin{proof}[Proof of Theorem \ref{thm:1.1}]
First we prove the uniqueness.
To this end, we suppose by contradiction that 
problem \eqref{eq:1.3} has at least two positive radial solutions.
Especially one has $G \ne \emptyset$.

Now by Proposition \ref{prop:2.2} (ii), 
$d_0= \inf (N \cup G)\in G$ is admissible and hence 
$d_0$ is strictly admissible by Proposition 2.3.
Thus by Proposition \ref{prop:2.2} (iii), 
there exists $\varepsilon>0$ such that
\begin{equation}
\label{eq:2.4}
 (d_0, d_0+\varepsilon) \subset N.
\end{equation}
Moreover taking $\varepsilon$ smaller if necessary, 
we may assume that $d \in (d_0, d_0+\varepsilon)$ is admissible.
In fact since $d_0$ is strictly admissible, it follows that
$\delta(r_0,d_0)<0$, $\delta'(r_0,d_0)<0$ and $g'\big( u(r_0,d_0) \big)<0$
for sufficiently large $r_0>0$.
Thus for $d$ sufficiently close to $d_0$, 
$\delta(r,d)$ has exactly one zero on $(0, r_0]$, 
$\delta(r_0,d)<0$, $\delta'(r_0,d)<0$ and $g' \big( u(r,d) \big)<0$
for $r\in (r_0, R(d))$.
If $\delta(r)=\delta(r,d)$ has a zero $r_1 \in (r_0,R(d))$, 
then $\delta$ has a negative minimum at $r_2 \in (r_0, r_1)$ and hence
\begin{equation*}
0 \leq  \delta''(r_2) = - \frac{N-1}{r_2} \delta'(r_2) - g'\big( u(r_2) \big) \delta(r_2)
=- g'\big( u(r_2) \big) \delta(r_2) <0.
\end{equation*}
This is a contradiction. Thus $d\in (d_0, d_0+\varepsilon)$ is admissible, that is, $\delta(r,d)$ has exactly one zero on
$(0,R(d))$ for any $d \in (d_0, d_0+\varepsilon)$.

Next since we assumed that the set $G$ has at least two elements, 
we have $(d_0, \tilde{b}) \cap G \not= \emptyset$ and hence we can define
\begin{equation*}
 d_1 := \inf (G \cup P) \cap (d_0, \tilde{b}).
\end{equation*}
Then from \eqref{eq:2.4}, it follows that $d_1> d_0$. 
Moreover by the definition, we can see that $(d_0, d_1) \subset N$.
Finally since $N$ and $P$ are open by Proposition \ref{prop:2.1} (i), 
we also have $d_1 \in G$.
Especially by Proposition \ref{prop:2.2} (i), $\delta(r,d_1)$ has
at least one zero on $(0,R)$. 

Let $\mathcal{N}(d)$ be the number of zeros of 
$\delta(r,d)$ on $(0,R(d))$.
Next we claim that $\mathcal{N}(d) \le 1$ for $d\in (d_0,d_1) \subset N$. 
Indeed we know that $\mathcal{N}(d)=1$ for any $d \in (d_0, d_0+\varepsilon)$.
Moreover by the continuity of zeros of $\delta(r,d)$ with respect to
$d$, $\mathcal{N}(d)$ changes only one by one as $d$ increases.
Especially by Proposition \ref{prop:2.3}, it cannot happen that
$\mathcal{N}(d)$ increases from 1 to 2, otherwise 
such a $d\in (d_0,d_1)$ is admissible but not strictly admissible.
This implies that $\mathcal{N}(d) \le 1$ for any $d\in (d_0,d_1)$.

Now suppose that $\delta(r,d_1)$ has exactly one zero. Then by Proposition 
\ref{prop:2.3}, it follows that $d_1$ is strictly admissible and hence 
$(d_1-\varepsilon', d_1) \subset P$ for some $\varepsilon'>0$ by Proposition \ref{prop:2.2} (iii). 
But this contradicts to the definition of $d_1$.
Next we suppose that $\delta(r,d_1)$ has at least two zeros.
Since zeros of $\delta(r,d)$ depend continuously on $d$, we have
$\mathcal{N}(d) \ge 2$ for $d<d_1$ sufficiently close to $d_1$.
This is a contradiction to $\mathcal{N}(d) \le 1$ for $d\in (d_0,d_1)$.
This completes the proof of the uniqueness.

Finally, let $\mathcal{L}=\Delta +g'(u): H^2(\mathbb{R}^N) \to L^2(\mathbb{R}^N)$ 
be the linearized operator around $u$,
and suppose that $\varphi_0$ is a nontrivial radial eigenfunction 
corresponding to 0-eigenvalue of $\mathcal{L}$, that is,
$$
\varphi_0''+\frac{N-1}{r}\varphi_0'+g'(u)\varphi_0=0, \ \varphi_0'(0)=0.$$
Since $\varphi_0$ and $\delta$ satisfy the same initial condition at the origin, 
$\varphi_0$ is a constant multiple of $\delta$.
Then by Proposition \ref{prop:2.3}, we have $\varphi_0 \to -\infty$
as $r \to \infty$ and hence $\varphi_0 \not\in H^1_{rad}(\mathbb{R}^N)$.
This completes the proof.
\end{proof}

Once we could obtain the uniqueness of the positive radial solution, 
one can show more precise information on the spectrum of $\mathcal{L}$. 
Indeed we can prove the following result.

\begin{corollary}\label{cor:2.4}
Assume (G1)-(G6) and let $\mathcal{L}=\Delta +g'(u)$ be the
linearized operator around the unique positive radial solution $u$ of \eqref{eq:1.1}.
Then the following properties hold.
\begin{itemize}
\item[\rm(i)] $\sigma(\mathcal{L})=\sigma_p(\mathcal{L}) \cup \sigma_e(\mathcal{L})$,
where $\sigma_p(\mathcal{L})$ and $\sigma_e(\mathcal{L})$
are the point spectrum and the essential spectrum of $\mathcal{L}$ respectively.
\item[\rm(ii)] 
$\sigma_e(\mathcal{L})=(-\infty,g'(0)]$ and 
$\sigma_p(\mathcal{L}) \subset (g'(0), \infty)$.
\item[\rm(iii)] If $\mu \in \sigma_p(\mathcal{L})$, 
then the corresponding eigenfunction $\varphi(x)$ satisfies
$$
| D^k \varphi(x)| 
\le C_{\epsilon} e^{-\sqrt{{-g'(0)+\mu+\varepsilon \over 2}}|x|}, \ x\in \mathbb{R}^N, |k| \le 2
$$
for any small $\varepsilon>0$ and some $C_{\epsilon}>0$.
\item[\rm(iv)] If $\mu \in \sigma_p(\mathcal{L}) \cap (0, \infty)$, 
then the corresponding eigenfunction is radially symmetric.
\item[\rm(v)] The principal eigenvalue $\mu_1(\mathcal{L}) >0$ 
is simple, and the corresponding eigenfunction $\varphi$ can be chosen to be positive.
\item[\rm(vi)] The second eigenvalue $\mu_2(\mathcal{L})$ is zero,
and the eigenspace associated with the eigenvalue $\mu =0$ is spanned by
$$
\left\{{\partial u \over \partial x_i}\,;\, i=1, \cdots ,N\right\}.
$$
\end{itemize}
\end{corollary}

\begin{proof}
Although the proof can be done in a similar way as \cite{BS},
we give the proof for the sake of completeness.

First, since $u$ is bounded and decays exponentially at infinity,
$\mathcal{L}$ is a self-adjoint operator from the domain ${\rm Dom}(\mathcal{L})=H^1(\mathbb{R}^N)$
to $L^2(\mathbb{R}^N)$, and a compact perturbation of $\Delta +g'(0)$.
Then (i), (ii) and (iii) follow by the standard spectral theory for self-adjoint operators. 

By the elliptic regularity theory and the exponential decays of $u^{(k)}(r)$,
it follows that $\frac{\partial u}{\partial x_i} \in H^2(\mathbb{R}^N)$ and
$\frac{\partial u}{\partial x_i}, i=1,\cdots, N$, are eigenfunctions of $\mathcal{L}$ 
corresponding to the eigenvalue $0$. 
Especially $\sigma_p(\mathcal{L}) \ne \emptyset$.
It is well-known that $\mu_1(\mathcal{L})$ is simple 
and the corresponding eigenfunction $\varphi_1$ has a constant sign.
Since $0$ is not a simple eigenvalue, 
it follows that $\mu_1(\mathcal{L})>0$ and hence (v) holds.

Next we show (iv). Let $\lambda_i$ and $\psi_i(\theta)$ with $\theta \in S^{N-1}$
be the eigenvalues and eigenfunctions of the Laplace-Beltrami operator on $S^{N-1}$.
Then it is well-known that
$$
0=\lambda_0<\lambda_1 = \cdots =\lambda_N=N-1 < \lambda_{N+1} \cdots,$$
$\psi_0(\theta)=1$ and $\{ \psi_i \}$ is a complete orthogonal basis of $L^2(S^{N-1})$.
Let $\varphi$ be a eigenfunction of $\mathcal{L}$ with eigenvalue $\mu \ge 0$. 
We define
$$
\varphi_i(r) := 
\int_{S^{N-1}} \varphi(r,\theta) \psi_i(\theta) \,d\theta.$$
Then from (iii), we have $\varphi_i'(0)=0$ and
\begin{equation} \label{eq:2.4.1}
| \varphi_i^{(k)} (r) | \le C e^{-\rho r}, \ 
r\in (0,\infty), \ k=0,1,2
\end{equation}
for some $C, \rho>0$. 
Moreover by a direct calculation, one has
\begin{equation} \label{eq:2.4.2}
\varphi_i'' +\frac{N-1}{r} \varphi_i' + \left( g'(u)-\frac{\lambda_i}{r^2} \right) \varphi_i
=\mu \varphi_i.
\end{equation}
On the other hand, $u'$ satisfies $u'<0$ and
\begin{equation} \label{eq:2.4.3}
u'''+\frac{N-1}{r}u''+ \left( g'(u)-\frac{N-1}{r^2} \right)u'=0.
\end{equation}
Multiplying \eqref{eq:2.4.2}, \eqref{eq:2.4.3} by $r^{N-1}u'$, $r^{N-1}\varphi_i$ respectively,
then subtracting and integrating over $(0,r)$, we get
\begin{align} \label{eq:2.4.4}
&r^{N-1} \big( u'(r)\varphi_i'(r)-u''(r)\varphi_i(r) \big)
+(N-1-\lambda_i)\int_0^r s^{N-3} u' \varphi_i \,ds \nonumber \\
&=\mu \int_0^r s^{N-1} u' \varphi_i \,dr.
\end{align}

We claim that $\varphi_i \equiv 0$ for $i \ge N+1$.
If not, we may assume that $\varphi_i(0)>0$. 
Let $r_i \in (0,\infty]$ be such that $\varphi_i(r)>0$ on $[0,r_i)$,
$\varphi_i(r_i)=0$ and $\varphi_i'(r_i) < 0$.
When $r_i < \infty$, it follows from \eqref{eq:2.4.4} that
$$
0< r_i^{N-1} u'(r_i) \varphi_i'(r_i)
+(N-1-\lambda_i) \int_0^{r_i} s^{N-3} u' \varphi_i \,ds
=\mu \int_0^{r_i} s^{N-1} u' \varphi_i \,ds \le 0,$$
which is a contradiction. On the other hand when $r_i=\infty$,
by the exponential decays of $u', u'', \varphi_i$ and $\varphi_i'$, we have
$$
\lim_{r \to \infty} 
\left\{ r^{N-1} \big( u'(r) \varphi_i'(r) -u''(r) \varphi_i(r) \big) \right\}=0.$$
Thus we get
$$
0<(N-1-\lambda_i) \int_0^{\infty} s^{N-3} u' \varphi_i \,ds
=\mu \int_0^{\infty} s^{N-1} u' \varphi_i \,ds \le 0.$$
This implies that $\varphi_i \equiv 0$ for $i \ge N+1$ as claimed.
Since $\{ \psi_i\}$ is a complete orthogonal basis of $L^2(S^{N-1})$, we have
\begin{equation} \label{eq:2.4.5}
\varphi(x)= \varphi(r,\theta)= \sum_{i=0}^N \varphi_i(r) \psi_i( \theta).
\end{equation}

Next we suppose that $\mu>0$. If $\varphi_i \not \equiv 0$, for some $i=1, \dots, N$, in a similar argument as above,
it follows from $\lambda_i=N-1$ that
$$
\varphi_i(r) >0 \quad \hbox{for} \ r\in (0,\infty)
.$$
Then taking $r \to \infty$ in \eqref{eq:2.4.4}, we get
$$
0= \mu \int_0^{\infty} s^{N-1} u' \varphi_i \,ds <0.$$
This is a contradiction and hence $\varphi_i \equiv 0$ for $i=1, \cdots, N$ if $\mu>0$.
Thus from \eqref{eq:2.4.5}, we obtain
$$
\varphi(x)=\varphi_0(r) \psi_0(\theta) = \varphi_0(r)$$
and hence $\varphi$ is radially symmetric. 
This completes the proof of (iv).

Finally we prove (vi). 
Let $\varphi$ be an eigenfunction of $\mathcal{L}$ with eigenvalue $\mu=0$.
From \eqref{eq:2.4.2} and $\lambda_0=0$, it follows that
$$
\varphi_0''+\frac{N-1}{r}\varphi_0'+g'(u)\varphi_0=0, \ 
\varphi_0'(0)=0.$$
Moreover from \eqref{eq:2.4.1}, we also have $\varphi_0 \in H^1_{rad}(\mathbb{R}^N)$.
Then by Theorem \ref{thm:1.1}, we get $\varphi_0 \equiv 0$.
Thus from \eqref{eq:2.4.5}, we obtain
$$
\varphi(x)= \sum_{i=1}^N \varphi_i(r) \psi_i(\theta)$$
and hence ${\rm dim} \mathop{\rm Ker}\,(\mathcal{L}) \le N$. 
On the other hand since $\frac{\partial u}{\partial x_i} \in \mathop{\rm Ker}\,(\mathcal{L})$
for $i=1,\cdots, N$, it follows that
$$
\mathop{\rm Ker}\,(\mathcal{L})={\rm span} \left\{ \frac{\partial u}{\partial x_i} \ ; \ 
i=1,\cdots, N \right\}.$$

To finish the proof, it suffices to show that $\mu_2(\mathcal{L})=0$. 
If $\mu_2(\mathcal{L})>0$, 
the corresponding eigenfunction $\varphi$ is radially symmetric from (iv).
This implies that $\varphi$ satisfies 
\begin{equation} \label{eq:2.4.6}
\varphi''+\frac{N-1}{r}\varphi' +\big( g'(u)-\mu_2 \big) \varphi =0, \ 
r\in (0,\infty), \ \varphi'(0)=0.
\end{equation}
Moreover from (iii), one can show that $\varphi$, $\varphi'$ decay exponentially and
$$
\frac{\varphi'}{\varphi} \to - \sqrt{ -g'(0)+\mu_2} \ \hbox{as} \ r \to \infty.$$
First we claim that $\varphi$ has a zero on $(0,\infty)$. 
To this end, let $\tilde{\varphi}$ be an eigenfunction of $\mathcal{L}$
corresponding to eigenvalue $\mu_1>0$. 
Then $\tilde{\varphi}$ decays exponentially and satisfies
$$
\tilde{\varphi}''+\frac{N-1}{r}\tilde{\varphi}'
+\big( g'(u)-\mu_1 \big) \tilde{\varphi}=0, \ r\in (0,\infty), \ \tilde{\varphi}'(0)=0.$$
Now we suppose that $\varphi \ne 0$ on $(0,\infty)$. 
Then we have the following identity:
$$
\left[ r^{N-1}\tilde{\varphi}\tilde{\varphi}'-r^{N-1}\tilde{\varphi}^2
\frac{\varphi'}{\varphi} \right]'
=(\mu_1-\mu_2) \tilde{\varphi}^2
+r^{N-1} \left( \tilde{\varphi}'-\tilde{\varphi} \frac{\varphi'}{\varphi} \right)^2
\ \hbox{on} \ (0,\infty).$$
In a similar way as the proof of Proposition \ref{prop:2.2} (i),
integrating over $[0,r]$ and passing a limit $r \to \infty$, we get
$$
(\mu_1-\mu_2) \int_0^{\infty} \tilde{\varphi}^2 \,ds \le 0.$$
However since $\mu_1$ is the principal eigenvalue, it follows that $\mu_1>\mu_2$,
which leads a contradiction. 
Thus there exists $r_0 \in (0,\infty)$ such that $\varphi(r_0)=0$. 

Next we compare $\varphi$ and $\delta$. 
First from $\mu_2>0$ and by Lemma \ref{lem:A.1} (b), it follows that $r_{\delta}<r_0$,
where $r_{\delta}$ is the unique zero of $\delta$.
Especially $\delta \ne 0$ on $(r_0,\infty)$. 
Then from \eqref{eq:2.3} and \eqref{eq:2.4.6}, one has the following identity:
$$
\left[ r^{N-1}\varphi\varphi'-r^{N-1}\varphi^2 \frac{\delta'}{\delta} \right]'
=\mu_2 \varphi^2 +r^{N-1} \left( \varphi'-\varphi \frac{\delta'}{\delta} \right)^2
\ \hbox{on} \ (r_0,\infty).$$
Arguing similarly as in the proof of Proposition \ref{prop:2.2} (i), we get 
$$
\limsup_{r \to \infty} r^{N-1} \left( \varphi\varphi'-r^{N-1}\varphi^2 
\frac{\delta'}{\delta} \right) \le 0.$$
Thus from $\varphi(r_0)=0$ and $\delta(r_0) \ne 0$, we obtain
$$
0< \mu_2 \int_{r_0}^{\infty} \varphi^2 \,ds \le 0.$$
This is a contradiction and hence $\mu_2=0$.
This completes the proof of Corollary \ref{cor:2.4}.
\end{proof}

Now we give some preliminaries to prove Proposition \ref{prop:2.3}.
For $s>0$ and $\lambda>0$, we define the {\it I-function} by
$$
 I(s, \lambda):=\lambda s g'(s) -(\lambda+2) g(s). 
$$
We note by the definition of $K_g(s)$ that 
\begin{equation}
\label{eq:2.5}
 I(s, \lambda)=
\lambda g(s) 
\left\{
K_g(s) - 1 - \frac{2}{\lambda}
\right\} \ \hbox{for} \ s \ne b, \tilde{b}.
\end{equation}
Moreover we put $\displaystyle K_\infty := \lim_{s \to \tilde{b}-} K_g(s)$.
Then from (G2) and (G4), $K_{\infty}$ is well-defined, 
$K_{\infty}\in [-\infty,\infty)$ and $K_{\infty}=-\infty$ if $\tilde{b}<\infty$.
We also observe that $\displaystyle \lim_{s \to b+} K_g(s)=\infty$ by (G2). 

\begin{lemma}\label{lem:2.5} \ Assume (G1)-(G5).
\begin{itemize}
\item[\rm(i)] 
Let $s^* :=\infty$ if $K_{\infty} \ge 1$ and $s^* :=
(K_g)^{-1}(1)\in (b,\tilde{b})$ if $K_{\infty}<1$.
Then for each $t \in (b, s^*)$, 
there exists a unique $\Lambda(t)>0$ such that
\begin{align*}
& I \big( s, \Lambda(t) \big) > 0 \text{ for } 0 < s < t, \\
& I \big( s, \Lambda(t) \big) < 0 \text{ for } t<s<\tilde{b}.
\end{align*}
Moreover the map $t \mapsto \Lambda(t)$ is continuous.
\item[\rm(ii)] 
Let $\tilde{\lambda}>0$, $c\in (b, \tilde{b})$ be given and
suppose that $I(c, \tilde{\lambda}) \geq 0$. Then 
\begin{equation*}
 I(s, \lambda) > 0 \text{ for all } s\in (b,c]
\ \hbox{and} \ \lambda>\tilde{\lambda}.
\end{equation*}
\item[\rm(iii)] 
If $K_ \infty < 1$, then 
$I(s, \lambda) < 0$ for all $s \in [s^*, \tilde{b})$ and $\lambda>0$. 
\end{itemize}
\end{lemma}

\begin{proof}
(i):
Since $K_g(s)$ is decreasing on $(b, \tilde{b})$ by (G4), 
it follows that $s^*$ is well-defined and $K_g(t) >1$ for $t \in (b,s^*)$.
Thus there exists a unique $\lambda>0$ such that
$K_g(t)=1+\frac{2}{\lambda}$.
Putting $\Lambda(t)=\frac{2}{K_g(t)-1}$, (i) holds.

(ii): By the assumption $I(c,\tilde{\lambda}) \geq 0$ and (G2), it follows that 
$$
K_g(c) - \frac{\tilde{\lambda}+2}{\tilde{\lambda}} \geq 0.
$$
From $\frac{\tilde{\lambda}+2}{\tilde{\lambda}} > \frac{\lambda+2}{\lambda}$, 
one has $K_g(c) - \frac{\lambda+2}{\lambda} > 0$. 
Since $K_g(s)$ is decreasing, the claim holds.

(iii):
Let $s \in [s^*,\tilde{b})$. Then one has from (G5) that $K_g(s) \leq  K_g(s^*) =1$
and $\frac{\lambda+2}{\lambda} > 1$. 
Thus from \eqref{eq:2.5}, the claim follows.
\end{proof}

\begin{remark}\label{rem:2.6}
If $K_{\infty}\ge 1$, Lemma \ref{lem:2.5} (i) shows that 
I-theorem in \cite{Mc} holds \rm{(}See \cite[Theorem 1]{Mc}\rm{)}.
Then Proposition \ref{prop:2.3} follows from Lemma 8 in \cite{Mc}.
\end{remark}

\section{Proof of the strict admissibility}

In this section, we prove Proposition \ref{prop:2.3}.
To this end, let $d \in N \cup G$ and assume that
$\delta(r,d)$ has exactly one zero in $(0,R)$.
Furthermore we suppose by contradiction that 
$\delta(R)=0$ when $d \in N$ and 
$\delta(r)$ stays bounded at infinity when $d \in G$. 
Then by Proposition \ref{prop:2.2} and from \eqref{eq:2.2}, 
we can show that $\delta'(R)>0$ if $d \in N$ and 
$\delta(r) \to 0$ exponentially as $r \to \infty$ if $d \in G$. 
(See \cite[Lemma 2, p. 497]{Mc}.)

To derive a contradiction, we define 
$$
v_\lambda(r) := v_{\lambda}(r,d)=r u'(r) + \lambda u(r) \ \hbox{for} \ \lambda>0.$$
We note that $v_{\lambda}(0)=\lambda u(0)>0$. 
Moreover one can see that $v_\lambda$ satisfies 
\begin{equation}
\label{eq:3.1}
v_\lambda'' + \frac{N-1}{r} v_\lambda' + g'(u) v_\lambda=I(u, \lambda) \text{ on } (0,R)
\end{equation}
and hence 
\begin{equation}
\label{eq:3.2}
v_\lambda'' + \frac{N-1}{r} v_\lambda' + \left[g'(u) - \frac{I(u, \lambda)}{v_\lambda}\right] v_\lambda=0 \ \hbox{for} \ v_\lambda \not=0.
\end{equation}
Our aim is to compare $\delta$ with $v_{\lambda}$, and our goal is to show the following lemma.

\begin{lemma}[Key lemma] \label{lem:3.1}
Let $r_{\delta} \in (0,R)$ be the unique zero of $\delta(r,d)$.
Then there exists $\lambda_0 >0$ such that
$$
v_{\lambda_0}(r_\delta)=0, \ 
v_{\lambda_0} <0 \text{ on } (r_\delta, R)
\ \hbox{and} \ 
v_{\lambda_0}'(r_{\delta})<0.
$$
\end{lemma}
As we will see later, we can prove Proposition \ref{prop:2.3} by 
Lemma \ref{lem:3.1}. (See also Remark \ref{rem:3.7} below.)
The proof of Lemma \ref{lem:3.1} consists of several lemmas.

\begin{lemma}\label{lem:3.2}
Let $\lambda>0$. Then we obtain the following properties.
\begin{itemize}
\item[\rm(i)] 
For sufficiently large $r>0$, it follows that $v_\lambda(r)<0$.
\rm{(}$v_{\lambda}(R)<0$ when $R < \infty$.\rm{)} 
Especially $v_\lambda$ has a zero on $(0,R)$.
\item[\rm(ii)] 
For sufficiently large $\lambda>0$, $v_\lambda$ does not have any zeros on $[0,r_\delta]$. 
Especially it follows that $v_\lambda>0$ on $[0, r_\delta]$.
\end{itemize}
\end{lemma}

\begin{proof}
(i): First we consider the case $d \in G$. 
By the definition of $v_{\lambda}$ and from \eqref{eq:2.2}, one has
$$
\lim_{r \to \infty} \frac{v_{\lambda}(r)}{u(r)} 
= \lim_{r \to \infty} \frac{r u'(r)}{u(r)} + \lambda 
= -\infty.
$$
Since $v_\lambda$ is continuous and $v_\lambda(0)>0$,
the claim holds. When $d\in N$, the claim also 
follows from $v_{\lambda}(R)=Ru'(R)<0$.

(ii): By the definition of $v_\lambda$, it follows that 
\begin{equation}
\label{eq:3.3}
 v_\lambda(r)= r u'(r) + \lambda u(r) \geq \inf_{r \in [0,r_\delta]} (r u'(r)) 
+ \lambda \inf_{r \in [0,r_\delta]} u(r)
\text{ for } r \in [0,r_\delta].
\end{equation}
Moreover we have
$$
 \inf_{r \in [0,r_\delta]} (r u'(r)) \in (-\infty,0) \ \hbox{and} \ 
\inf_{r \in [0,r_\delta]} u(r) >0.
$$
Thus the r.h.s of \eqref{eq:3.3} is positive for sufficiently large $\lambda$.
\end{proof}

\begin{lemma}
\label{lem:3.3}
Let $r_{\delta}$ be the unique zero of $\delta(r)$ and suppose that
$\delta(r) \to 0$ as $r \to R$. Then it follows that 
$b< u(r_{\delta}) < s^*$, where $s^*$ is a constant defined in Lemma \ref{lem:2.5} (i). 
\end{lemma}

\begin{proof}
First we show that $b<u(r_{\delta})$. To this aim, suppose by contradiction
that $u(r_{\delta})\le b$. Then it follows that $u(r) < b$ for $r_{\delta}<r<R$
and hence $g(u)<0$ on $(r_{\delta},R)$ by (G2).

Next from \eqref{eq:2.1} and \eqref{eq:2.3}, one has
\begin{align}\label{eq:3.4}
\left[ (r^{N-1}u')(r^{N-1}\delta') \right]'
&= -r^{2N-2} \big(
g(u)\delta' +g'(u)u\delta' \big) \nonumber \\
&= -r^{2N-2} \big( g(u) \delta \big)'.
\end{align}
Integrating \eqref{eq:3.4} over $[r_{\delta},r]$ and using the integration by parts, we get
\begin{align*}
&r^{2N-2}u'(r)\delta'(r)-r_{\delta}^{2N-2} u'(r_{\delta})\delta'(r_{\delta})\\
&=-r^{2N-2}g\big( u(r) \big) \delta(r) 
+(2N-2) \int_{r_{\delta}}^r s^{2N-3} g \big( u(s) \big) \delta (s) \,ds.
\end{align*}

Now suppose that $d \in G$. 
Then since $u$, $u'$ and $\delta$ decay exponentially as $r \to \infty$,
we can pass the limit $r \to \infty$ to obtain
$$
0>-r_{\delta}^{2N-2} u'(r_{\delta})\delta'(r_{\delta})
=(2N-2) \int_{r_{\delta}}^{\infty}
s^{2N-3} g \big( u(s) \big) \delta (s) \,ds.$$
However since $g(u)<0$ and $\delta<0$ on $(r_{\delta},\infty)$,
this is a contradiction.
Next let $d\in N$. In this case, one has $u(R)=0$, $u'(R)<0$, 
$\delta(R)=0$ and $\delta'(R)>0$.
Taking $r=R$, we get
\begin{align*}
0&>R^{2N-2} u'(R)\delta'(R) 
-r_{\delta}^{2N-2} u'(r_{\delta})\delta'(r_{\delta}) \\
&=(2N-2) \int_{r_{\delta}}^{R}
s^{2N-3} g \big( u(s) \big) \delta (s) \,ds>0.
\end{align*}
This is a contradiction again and hence $b<u(r_{\delta})$.

Next we prove $u(r_{\delta})<s^*$. By the definition of $s^*$, 
it suffices to consider the case $K_{\infty}<1$.
Suppose by contradiction that $u(r_{\delta}) \ge s^*$.
Since $u$ is decreasing, one has $u \ge s^*$ on $(0,r_\delta)$.
Moreover as we have mentioned in Section 1,
we may assume that $u(0)<\tilde{b}$ by the Maximum Principle.
Thus by Lemma \ref{lem:2.5} (iii), we get
$$
 I(u(r), \lambda) < 0 \text{ for all } r \in (0, r_\delta) \ 
\hbox{and} \ \lambda >0.
$$
Moreover taking large $\lambda>0$, we have by Lemma \ref{lem:3.2} (ii)
that $v_{\lambda}>0$ on $(0,r_{\delta})$. Thus we obtain
$$
 g'(u) < g'(u) - \frac{I(u, \lambda)}{v_\lambda} \text{ on } (0, r_\delta)
\ \hbox{for sufficiently large} \ \lambda.
$$
From \eqref{eq:2.3}, \eqref{eq:3.2} and $\delta(r_{\delta})=0$,
we can apply the Sturm Comparison Principle (Lemma \ref{lem:A.1} (b) below)
to show that $v_\lambda$ has a zero on $(0, r_\delta)$.
This contradicts to Lemma \ref{lem:3.2} (ii) and hence the claim holds.
\end{proof}

Now by Lemma \ref{lem:3.3}, 
we can apply Lemma \ref{lem:2.5} (i) with $t=u(r_{\delta})$.
Putting $\underline{\lambda}=\Lambda \big( u(r_{\delta}) \big)$, we get
$$
I(s, \underline{\lambda}) < 0 \text{ for } s > u(r_\delta)
 \ \hbox{and} \ 
I(s, \underline{\lambda}) > 0 \text{ for } 0 < s < u(r_\delta).
$$
Then by Lemma \ref{lem:2.5} (ii) and the continuity of $u$, we also have 
$$
 I(s, \lambda) > 0 \text{ for } 
\left\{
\begin{array}{l}
0 < s < u(r_\delta) \ \hbox{and} \ 
\lambda \geq \underline{\lambda},\\
0<s \le u(r_{\delta}) \ \hbox{and} \
\lambda > \underline{\lambda}.
\end{array}
\right.
$$
Finally since $u$ is decreasing, we obtain
\begin{equation}
\label{eq:3.5}
 I(u(r), \underline{\lambda}) < 0 \text{ for } 0 < r < r_\delta, \ 
 I(u(r), \lambda) > 0 \text{ for } 
\left\{
\begin{array}{l}
r_\delta<r<R, \ 
\lambda \geq \underline{\lambda},\\
r_{\delta} \le r <R, \ \lambda > \underline{\lambda}.
\end{array}
\right.
\end{equation}
Next we investigate the sign of $v_{\lambda}$ near its zero.

\begin{lemma}\label{lem:3.4} We obtain the following properties.
\begin{itemize}
\item[\rm(i)] Let $\lambda \geq \underline{\lambda}$
and suppose that $v_\lambda$ has a zero $r_0 \in (r_\delta, R)$.
Then it follows that either $v_{\lambda}'(r_0) \ne 0$, 
or $v_{\lambda}'(r_0)=0$ and $v_{\lambda}''(r_0)>0$.
In other words, if $v_{\lambda}$ is negative before $r_0$, 
then $v_{\lambda}$ must be positive after $r_0$.
\item[\rm(ii)]
Suppose that $v_{\underline{\lambda}}$ has a zero $r_0 \in (0, r_\delta)$.
Then it follows that either $v_{\underline{\lambda}}'(r_0) \ne 0$, 
or $v_{\underline{\lambda}}'(r_0)=0$ and $v_{\underline{\lambda}}''(r_0)<0$.
In other words, if $v_{\underline{\lambda}}$ is positive before $r_0$, 
then $v_{\underline{\lambda}}$ must be negative after $r_0$.
\end{itemize} 
\end{lemma}

\begin{proof}
Assume that $v_{\lambda}(r_0)=0$ and $v_{\lambda}'(r_0)=0$ for some
$r_0 \in (r_{\delta},R)$.
Then from \eqref{eq:3.1} and \eqref{eq:3.5}, one has
$$
v_\lambda''(r_0) = I(u(r_0), \lambda)>0.
$$
Thus the claim holds. (ii) can be shown in a similar way.
\end{proof}

\begin{lemma}\label{lem:3.5}
Let $\lambda \geq \underline{\lambda}$. 
Then $v_\lambda$ has at most one zero on $(r_\delta, R)$.
More precisely if $v_\lambda$ has a zero $r_0 \in (r_\delta, R)$, 
then it follows that
$$
 v_\lambda > 0 \text{ on } (r_\delta, r_0), \ 
 v_\lambda < 0 \text{ on } (r_0, R) \ \hbox{and} \ 
v_\lambda'(r_0) <0.$$
\end{lemma}

\begin{proof}
First we prove the following claim:
\begin{claim*}
If $\hat{r}_0 \in (r_{\delta},R)$ is a zero of $v_{\lambda}$,
then $v_{\lambda}<0$ on some right neighborhood of $\hat{r}_0$.
\end{claim*} 
Indeed suppose by contradiction that $v_{\lambda}(r)>0$ for $r>\hat{r}_0$ near $\hat{r}_0$.
Then by Lemma \ref{lem:3.2} (i), $v_{\lambda}$ has a next zero $\bar{r}_0
\in (\hat{r}_0,R)$ and $v_{\lambda}>0$ on $(\hat{r}_0, \bar{r}_0)$. 
Thus from \eqref{eq:3.5}, one has
$$
 g'(u) > g'(u) - \frac{I(u, \lambda)}{v_{\lambda}} 
\text{ on } (\hat{r}_0, \bar{r}_0).$$
In this case, we can apply Lemma \ref{lem:A.1} (a) to conclude that $\delta$ must have a zero on 
$(\hat{r}_0, \bar{r}_0) \subset (r_{\delta},R)$.
This contradicts to the assumption which we made in the beginning of this section
and hence the claim holds.

Now we assume that there exists $r_0 \in (r_{\delta},R)$ such that
$v_{\lambda}(r_0)=0$. Then by the claim above, it follows that
$v_{\lambda}<0$ on a right neighborhood of $r_0$. 
Moreover if $v_{\lambda}$ has a next zero $r_1>r_0$, 
then $v_{\lambda}$ must change its sign from negative to positive by Lemma \ref{lem:3.4} (i).
This contradicts to the above claim provided $\hat{r_0}=r_1$ and hence
\begin{equation}
\label{eq:3.6}
v_\lambda<0 \text{ on } (r_0, R).
\end{equation}

Next we observe from \eqref{eq:3.6} and by Lemma \ref{lem:3.4} (i) that 
$v_{\lambda}>0$ on a left neighborhood of $r_0$. 
If there exists $r_2 \in (r_\delta, r_0)$ such that $v_{\lambda}(r_2)=0$,
it follows that $v_{\lambda} >0$ on 
$(r_2,r_0) \subset (r_{\delta},R)$. 
Then by applying the above claim with $\hat{r}=r_2$, we obtain a contradiction.
This implies that $v_{\lambda}>0$ on $(r_{\delta},r_0)$ and $v_{\lambda}'(r_0)<0$.
\end{proof}

\begin{lemma}
\label{lem:3.6}
$v_{\underline{\lambda}}$ has a unique zero $\underline{r} \in (0, r_\delta)$, that is,
$$
v_{\underline{\lambda}} > 0 \text{ on } (0, \underline{r}) \ \hbox{and} \ 
v_{\underline{\lambda}} < 0 \text{ on } (\underline{r}, R).$$
\end{lemma}

\begin{proof}
First we show that 
$v_{\underline{\lambda}}$ has at least one zero on $(0,r_\delta)$.
Indeed if $v_{\underline{\lambda}}>0$ on $(0,r_\delta)$, 
we have from \eqref{eq:3.5} that
$$
 g'(u) < g'(u) - \frac{I(u, \underline{\lambda})}{v_{\underline{\lambda}}} \text{ on } (0, r_\delta).
$$
Then by Lemma \ref{lem:A.1} (b), $v_{\underline{\lambda}}$ must have a zero on $(0,r_\delta)$.
This is a contradiction and hence the existence of a zero 
$\underline{r} \in (0, r_\delta)$ holds.
Moreover we may assume that $\underline{r}$ is the first zero of $v_{\underline{\lambda}}$.

Now by Lemma \ref{lem:3.4} (ii), $v_{\underline{\lambda}}<0$ on a right neighborhood of $\underline{r}$. 
If $v_{\underline{\lambda}}$ has a next zero $r_1 \in (\underline{r}, r_\delta]$, 
it follows that $v_{\underline{\lambda}} < 0$ on $(\underline{r}, r_1)$. 
From \eqref{eq:3.5}, we get
$$
 g'(u) > g'(u) - \frac{I(u, \underline{\lambda})}{v_{\underline{\lambda}}} \text{ on } (\underline{r}, r_1).$$
Then by Lemma \ref{lem:A.1} (a), $\delta$ must have a zero on $(\underline{r},r_1)$. 
This contradicts to the assumption of Proposition \ref{prop:2.3} and hence 
$v_{\underline{\lambda}} <0$ on $(\underline{r}, r_\delta]$. 
In this case, $v_{\underline{\lambda}}$ cannot have a zero on $(r_\delta, R)$
by Lemma \ref{lem:3.5}. This completes the proof.
\end{proof}

Now by Lemma \ref{lem:3.2} (ii), we can choose large $\overline{\lambda} > \underline{\lambda}$
so that $v_{\overline{\lambda}}>0$ on $[0,r_\delta]$. 
Then by Lemmas \ref{lem:3.2} (i) and \ref{lem:3.5}, 
$v_{\overline{\lambda}}$ has a unique zero $\overline{r} \in (r_\delta ,R)$.

To prove the key lemma, we define
$$
 \theta(r) := - \frac{r u'(r)}{u(r)}
$$
as in \cite{Kw}. Then one has
\begin{gather}
\theta(r) - \lambda = -\frac{ v_\lambda(r)}{u(r)},
\label{eq:3.7} \\
\theta'(r) 
=
\frac{-(r u'(r))' - \theta(r) u'(r)}{u(r)} 
= - \frac{v_\lambda'(r)}{u(r)} \ 
\hbox{for} \ r, \lambda \ \hbox{with} \ \theta(r)=\lambda.
\label{eq:3.8}
\end{gather}
Now we are ready to prove the key Lemma. 

\begin{proof}[Proof of Lemma \ref{lem:3.1}]
From \eqref{eq:3.7}, \eqref{eq:3.8} and by Lemma \ref{lem:3.5}, it follows that
\begin{equation}
\label{eq:3.9}
\theta'(\overline{r}) >0 \ \hbox{and} \ 
\theta(r) > \theta(\overline{r}) = \overline{\lambda} \text{ for } r > \overline{r}.
\end{equation}
Moreover by Lemma \ref{lem:3.6}, we also have
\begin{equation}
\label{eq:3.10}
\theta(r) > \underline{\lambda} \ \hbox{for} \ r > \underline{r}
\ \hbox{and hence} \ 
\theta(r_\delta) > \underline{\lambda}.
\end{equation}

Next we claim that $\theta'>0$ on $(r_\delta, \overline{r}]$.
Indeed suppose by contradiction that 
$\theta'( \hat{r})=0$ for some $\hat{r} \in (r_\delta, \overline{r})$.
Putting $\hat{\lambda}:=\theta(\hat{r})$, it follows from \eqref{eq:3.7} that
$v_{\hat{\lambda}}(\hat{r})=0$. Moreover since 
$\hat{r}>r_{\delta} >\underline{r}$, we have from \eqref{eq:3.10} that
$\hat{\lambda}>\underline{\lambda}$. 
Thus by Lemma \ref{lem:3.5}, we get $v_{\hat{\lambda}}'(\hat{r})<0$
and hence $\theta'(\hat{r})>0$ from \eqref{eq:3.8}.
This is a contradiction and the claim holds.

Now from \eqref{eq:3.9}, \eqref{eq:3.10} and $\theta'>0$ on $(r_\delta, \overline{r}]$, 
it follows that
$$
\theta(r) > \theta(r_\delta) > \underline{\lambda} \text{ for all } r>r_\delta.
$$
Putting $\lambda_0:= \theta(r_\delta)$, 
we have from \eqref{eq:3.7} that 
$$
v_{\lambda_0}(r_{\delta})=0 \ \hbox{and} \ 
v_{\lambda_0}<0 \ \hbox{on} \ (r_{\delta},R).$$
Finally we show that $v_{\lambda_0}'(r_\delta)<0$.
If not, one has $v_{\lambda_0}'(r_\delta)=0$.
Then from \eqref{eq:3.1} and \eqref{eq:3.5}, it follows that
$$
 v_{\lambda_0}''(r_\delta) =I(u(r_\delta), \lambda_0) >0.$$
This implies that $v_{\lambda_0}>0$ on a right neighborhood of $r_{\delta}$,
which contradicts to the fact $v_{\lambda_0}<0$ on $(r_{\delta},R)$.
Thus we obtain $v'_{\lambda_0}(r_{\delta})<0$.
\end{proof}

Finally we prove Proposition \ref{prop:2.3} by using Lemma \ref{lem:3.1}.

\begin{proof}[Proof of Proposition \ref{prop:2.3}]
As we have mentioned in the beginning of this section, 
we assume that $\delta(r,d)$ has exactly one zero $r_{\delta} \in (0,R)$
and suppose by contradiction that 
$\delta(R)=0$ when $d \in N$ and 
$\delta(r)$ stays bounded at infinity when $d \in G$. 
Then it follows that 
$\delta(r) \to 0$ exponentially as $r \to \infty$ if $d \in G$. 

First we suppose that $d \in N$. 
In this case, one has by Lemma \ref{lem:3.1} and \eqref{eq:3.5} that
$$
g'(u) < g'(u)-\frac{I(u,\lambda_0)}{v_{\lambda_0}}
\ \hbox{on} \ (r_{\delta},R).$$
Since $\delta(r_{\delta})=\delta(R)=0$, we can apply Lemma \ref{lem:A.1}
to show that $v_{\lambda_0}$ has a zero on $(r_{\delta},R)$.
But this contradicts to Lemma \ref{lem:3.1}.

Next we consider the case $d \in G$. We put
$$
\varphi(r):= r^{N-1} g'(u(r)) \ \hbox{and} \ 
\psi(r):=r^{N-1}
\left(
g'(u(r)) - \frac{I(u(r), \lambda_0)}{v_{\lambda_0}(r)}
\right).
$$
Then from \eqref{eq:3.5} and by Lemma \ref{lem:3.1}, one has
$\varphi< \psi$ on $(r_\delta, \infty)$.
Moreover since $v_{\lambda_0} \ne 0$ on $(r_{\delta},\infty)$, 
we can rewrite \eqref{eq:2.3} and \eqref{eq:3.2} as
$$
(r^{N-1} \delta')' + \varphi \delta=0 \ \hbox{and} \ 
(r^{N-1} v_{\lambda_0}')' + \psi v_{\lambda_0}=0
$$
respectively. By a direct computation, one can obtain the following Picone identity holds:
$$
\left[
r^{N-1} \delta \delta' - r^{N-1} \delta^2 \frac{v_{\lambda_0}'}{v_{\lambda_0}}
\right]'
=
(\psi-\varphi) \delta^2 
+ r^{N-1}\left(\delta' - \delta \frac{v_{\lambda_0}'}{v_{\lambda_0}}\right)^2 
\ \hbox{on} \ (r_{\delta},\infty).
$$
Integrating it over $[r,t] \subset (r_{\delta},\infty)$, we get
\begin{align}\label{eq:3.11}
\int_r^t ( \psi-\varphi ) \delta^2 \,ds 
&\le 
\left(
t^{N-1} \delta(t) \delta'(t) - t^{N-1} \delta^2(t) \frac{v_{\lambda_0}'(t)}{v_{\lambda_0}(t)}
\right) \nonumber \\
&\quad -\left(
r^{N-1} \delta(r) \delta'(r) - r^{N-1} \delta(r) v_{\lambda_0}'(r)
 \frac{ \delta(r)}{v_{\lambda_0}(r)}
\right).
\end{align}

Next from \eqref{eq:2.1}, \eqref{eq:2.2} and by the definition of $v_{\lambda_0}$,
it follows that 
\begin{align*}
 \frac{v_{\lambda_0}'(t)}{v_{\lambda_0}(t)} 
&=
\frac{u'(t) + t u''(t) + \lambda_0 u'(t)}{t u'(t) + \lambda_0 u(t)} \\
&=
\frac{\displaystyle
\frac{2-N+\lambda_0}{t} \frac{u'(t)}{u(t)} - \frac{g(u(t))}{u(t)}
}{\displaystyle
\frac{u'(t)}{u(t)} + \frac{\lambda_0}{t}
} \to 
 - \sqrt{-g'(0)} \text{ as } t\to \infty.
\end{align*}
Since $\delta$ decays exponentially at infinity, we have
\begin{equation}\label{eq:3.12}
\lim_{t \to \infty}
\left(
t^{N-1} \delta(t) \delta'(t) - t^{N-1} \delta^2(t) \frac{v_{\lambda_0}'(t)}{v_{\lambda_0}(t)}
\right)=0.
\end{equation}
Moreover since $\delta(r_\delta)=0$ and 
$v_{\lambda_0}(r_\delta)=0$, 
we can apply l'H\^opital's rule to obtain
$$
 \lim_{r \to r_\delta+} \frac{\delta(r)}{v_{\lambda_0}(r)}
= \frac{\delta'(r_\delta)}{v_{\lambda_0}'(r_\delta)}.
$$
Thus we also have
\begin{equation}\label{eq:3.13}
\lim_{ r\to r_{\delta}+}
\left(
r^{N-1} \delta(r) \delta'(r) - r^{N-1} \delta(r) v_{\lambda_0}'(r)
 \frac{ \delta(r)}{v_{\lambda_0}(r)}
\right)=0.
\end{equation}
Now from \eqref{eq:3.11}, \eqref{eq:3.12} and \eqref{eq:3.13}, 
it follows that
$$
\int_{r_{\delta}}^{\infty} \big( \psi(s) -\varphi(s) \big)
\delta^2 (s) \,ds \le 0.$$
This is a contradiction because $\varphi < \psi$ on $(r_{\delta},\infty)$
and hence the proof is complete.
\end{proof}

\begin{remark}\label{rem:3.7}
As we have noted in Remark \ref{rem:2.6}, 
we can obtain Proposition \ref{prop:2.3} by the previous result in 
\cite{Mc} when $K_{\infty} \ge 1$.
Actually the key of the proof of the strict admissibility
was to construct a comparison function $v_{\lambda}$ having the same zero with $\delta$.
In \cite{Mc}, this construction was done by applying the I-theorem.
Our argument shows that we can construct such $v_{\lambda}$ without using the I-theorem.

We also observe that $K_{\infty} \ge 1$ is equivalent to the fact that
$g(s)$ has superlinear or asymptotically linear growth at infinity.
Since we were able to obtain Proposition \ref{prop:2.3} even if 
$K_{\infty}<1$, our result can cover the case where
the nonlinearity may have sublinear growth or may be negative at infinity.
\end{remark}

\section{Application to ground states of modified Schr\"odinger equations}

In this section, we study the uniqueness 
of positive radial solutions for the quasilinear elliptic problem of the form:
\begin{equation}\label{eq:4.1}
 -\mathop{\rm div} \big( a(u) \nabla u \big) + \frac{1}{2} a'(u) |\nabla u|^2 
= h(u) \text{ in } \mathbb{R}^N.
\end{equation}
As an application, we will also show the uniqueness and the non-degeneracy of
ground states for modified Schr\"odinger equations in Subsection 4.4.

We impose the following conditions on the nonlinear term $h(t)$ and
the quasilinear term $a(t)$.
\begin{itemize}
 \item[\rm(H1)] $h \in C^1[0,\infty)$, $h(0)=0$ and $h'(0)<0$.
 \item[\rm(H2)] There exist $\beta >0$ and $\tilde{\beta} \in (\beta ,\infty]$
such that $h(t)<0$ for $t \in (0,\beta)$, $h(t)>0$ for $t \in (\beta, \tilde{\beta})$ 
and $h'(\beta)>0$.
If $\tilde{\beta}<\infty$, then $h(t)<0$ on $(\tilde{\beta},\infty)$
and $h'(\tilde{\beta})<0$.
 \item[\rm(H3)] There exists $\tilde{\zeta}>0$ such that
$H(\tilde{\zeta})=\int_0^{\tilde{\zeta}} h(t) \,dt >0$.
 \item[\rm(H4)] $K_h(t)$ is decreasing on $(\beta,\tilde{\beta})$
and $K_h(t) \le 1$ on $(0,\beta)$.
 \item[\rm(H5)] There exists $\ell>0$ such that 
$
\left\{ 
\begin{array}{l}
\displaystyle \limsup_{t \to \infty} \frac{h(t)}{t^{ \frac{ (\ell+1)N+2}{N-2}}} \le 0 \ \hbox{if} \ N \ge 3, \vspace{0.5em} \\
\displaystyle \limsup_{t \to \infty} \frac{h(t)}{e^{ \tilde{\alpha} t^{\ell+2}}} \le 0 \ \hbox{for any} \ \tilde{\alpha}>0 \ \hbox{if} \ N=2.
\end{array}
\right.$
\end{itemize}
\begin{itemize}
 \item[\rm(A1)] $a \in C^2[0,\infty)$ and $\displaystyle \inf_{t \geq 0} a(t) >0$.
 \item[\rm(A2)] $a'(t) \geq 0$ on $(0,\infty)$.
 \item[\rm(A3)] $K_a(t)$ is non-decreasing on $[\beta,\tilde{\beta})$ and 
$K_a(t) \leq K_a(\beta)$ for $t\in (0,\beta)$.
 \item[\rm(A4)] There exists $a_\infty>0$ such that 
$\displaystyle \lim_{t \rightarrow \infty} \frac{a(t)}{t^{\ell}}= a_\infty$ ($\ell>0$ is defined in (H5)).
\end{itemize}
Here $K_h$ and $K_a$ are the growth functions of $h$ and $a$ respectively,
which are defined in \eqref{eq:1.2}.
We observe that if
$K_a(t)$ is non-decreasing on $[0,\infty)$, then (A3) automatically holds.

A typical example of $a(t)$ is given by $a(t)=1+\kappa |t|^{\ell}$ for
$\kappa>0$ and $\ell \ge 2$. 
Moreover by direct computations, we can see that 
$a(t)=1+|t|^{\ell_1}+|t|^{\ell_2}$ for $0<\ell_1<\ell_2$ or
$a(t)=|t|^2+e^{-ct^2}$ for $0<c \le 1$ satisfy (A1)-(A4).
As for the nonlinear term $h(t)$, typical examples are given by:
\begin{itemize}
\item Power nonlinearity:
$$h(t)=-\lambda t+|t|^{p-1}t \ \  \hbox{for} \ \lambda>0, \ell>0 
\ \hbox{and} \ 
\displaystyle 
\left\{
\begin{array}{ll}
1<p< \frac{(\ell+1)N+2}{N-2} & \hbox{if} \ N \ge 3,\\
1<p<\infty &\hbox{if} \ N=2.
\end{array}
\right.$$
\item Defocusing cubic-focusing quintic nonlinearity:
$$
h(t)=-t-t^3+t^5 \ \hbox{for} \ 
\left\{
\begin{array}{ll}
\ell > 4- \frac{12}{N} & \hbox{if} \ N \ge 4,\\
\ell >0 & \hbox{if} \ N=2,3.
\end{array}
\right.
$$
\item Focusing cubic-defocusing quintic nonlinearity:
$$
h(t)=-t+ct^3-t^5 \ \hbox{for} \ c> \frac{4\sqrt{3}}{3} 
\ \hbox{and} \ N \ge 2.
$$
\item Nagumo type nonlinearity:
$$
h(t)=t(t-c)(1-t) \ \hbox{for} \ 0<c<\frac{1}{2} 
\ \hbox{and} \ N \ge 2.
$$
\end{itemize}
By elementary calculations, one can show that these nonlinearities fulfill (H1)-(H5).
Similar statements also hold for quadratic-cubic nonlinearities: $h(t)=-t \pm c|t|t \mp |t|^2t$.

In this setting, we obtain the following result.
\begin{theorem}
\label{thm:4.1}
Assume (A1)--(A4) and (H1)--(H5). 
Then \eqref{eq:4.1} has a unique positive radial solution $u\in C^2(\mathbb{R}^N)$.
Moreover let $L:H^2(\mathbb{R}^N) \to L^2(\mathbb{R}^N)$ be the linearized operator of \eqref{eq:4.1} which is given by
\begin{align*}
L(\phi) =&
-a(u)\Delta \phi -a'(u) \nabla u \cdot \nabla \phi-\frac{1}{2}a''(u)|\nabla u|^2 \phi 
-a'(u)\Delta u \phi -h'(u)\phi.
\end{align*}
Then $\mathop{\rm Ker}\,(L) \big|_{H^1_{rad}(\mathbb{R}^N)} = \{ 0 \}$. 
\end{theorem}

\subsection{Dual approach and auxiliary lemmas}

In this subsection, we introduce a {\it dual approach} of \eqref{eq:4.1}
and prepare some auxiliary lemmas. 

To this aim, let $f(s)$ be a unique solution of the following ODE:
\begin{equation}\label{eq:4.2}
f'(s)= \frac{1}{\sqrt{a \big( f(s) \big)}} \text{ for } s>0, \quad f(0)=0.
\end{equation}
From (A1), we can see that $f$ is well-defined and $f\in C^2[0,\infty)$.
We also extend $f(s)$ as an odd function for $s<0$.
Then we have the following.

\begin{lemma}\label{lem:4.2}
$f(s)$ satisfies the following properties.
\begin{itemize}
\item[\rm(i)] $f>0$ and $f'>0$ on $(0,\infty)$.
Especially $f$ is monotone on $(0,\infty)$ and hence the inverse $f^{-1}$ exists.
\item[\rm(ii)] $\displaystyle s=\int_0^{f(s)} \sqrt{a(\tau)} \,d\tau$
and $\displaystyle f''(s)=- \frac{a' \big( f(s) \big)}
{2  a^2 \big( f(s) \big) }$.
\item[\rm(iii)] $f(s)$ has the following asymptotic behavior.
\begin{align*}
&\lim_{ s \to \infty} 
\frac{f(s)}{s^{2\over \ell +2}} =
\left( \frac{\ell+2}{2 \sqrt{a_\infty}} \right)^{2\over \ell+2}, \quad
\lim_{ s \to \infty} 
\frac{f'(s)}{s^{\frac{2}{\ell +2}-1}} =
\frac{2}{\ell+2}\left( \frac{\ell+2}{2 \sqrt{a_\infty}} \right)^{2\over \ell+2},\\
&\lim_{s \to \infty} \frac{sf'(s)}{f(s)}
= \frac{2}{\ell+2}.
\end{align*}
\end{itemize}
\end{lemma}

\begin{proof}
(i) and (ii) follow from \eqref{eq:4.2}.
Moreover from (A4), we can show that (iii) holds.
(For the proof, we refer to \cite{GS}.)
\end{proof}

Now we consider the following semilinear elliptic problem:
\begin{equation}
\label{eq:4.3}
- \Delta v = h \big( f(v) \big) f'(v) \text{ in } \mathbb{R}^N,
\end{equation}
which we call a {\it dual problem} of \eqref{eq:4.1}. 
We also define the linearized operator 
$\tilde{L}:H^2(\mathbb{R}^N) \to L^2(\mathbb{R}^N)$ of \eqref{eq:4.3} by 
\begin{equation*}
\tilde{L}(\tilde{\phi}):=-\Delta \tilde{\phi}
-\left( h'\big( f(v) \big) f'(v)^2 +h \big(f(v) \big) f''(v) \right)\tilde{\phi}. 
\end{equation*}
Then we have the following relation between \eqref{eq:4.1} and \eqref{eq:4.3}.

\begin{lemma}\label{lem:4.3} 
Let $X=\{ u\in H^1(\mathbb{R}^N) \ ; \ a(u)|\nabla u|^2 \in L^1(\mathbb{R}^N) \}$.
\begin{itemize}
\item[\rm(i)] 
$u\in X\cap C^2(\mathbb{R}^N)$ is a positive radial solution of \eqref{eq:4.1}
if and only if $v=f^{-1}(u)\in H^1\cap C^2(\mathbb{R}^N)$ is a positive radial solution of \eqref{eq:4.3}.
\end{itemize}

\noindent Let $u\in X\cap C^2(\mathbb{R}^N)$ is a positive solution of \eqref{eq:4.1} and put $v=f^{-1}(u)$. Then

\begin{itemize}
\item[\rm(ii)] 
For $\phi\in H^2(\mathbb{R}^N)$, let $\tilde{\phi}=\sqrt{a(u)}\phi$. 
Then $\tilde{\phi} \in H^2(\mathbb{R}^N)$ and the following identity holds: 
\begin{equation*}
\tilde{L}(\tilde{\phi})={1\over \sqrt{a(u)}}L(\phi).
\end{equation*}
\item[\rm(iii)] 
$\phi \in \mathop{\rm Ker}\, (L)$ if and only if $\tilde{\phi}=\sqrt{a(u)}\phi \in \mathop{\rm Ker}\,(\tilde{L})$.
\item[\rm(iv)] 
$\displaystyle \mathop{\rm Ker}\,(L)= \mathop{\rm span} \left\{ {\partial u \over \partial x_i} \right\}_{i=1}^N$ 
if and only if 
$\displaystyle \mathop{\rm Ker}\,(\tilde{L})= \mathop{\rm span} \left\{ {\partial v \over \partial x_i} \right\}_{i=1}^N$.
\item[\rm(v)] 
$u$ is non-degenerate if and only if $v$ is non-degenerate.
\end{itemize}
\end{lemma}
For the proof of (i), we refer to \cite{GS}.
The proof of (ii)-(v) can be found in \cite{ASW1}.
(See also \cite[Lemmas 2.7 and 2.8]{AW2}.)
By Lemma \ref{lem:4.3}, it suffices to study the uniqueness and the non-degeneracy of positive radial solutions of \eqref{eq:4.3}.

\begin{remark}\label{rem:4.4}
Let us consider the most typical case:
$$
h(t)=-\lambda t+|t|^{p-1}t \ \hbox{and} \ 
a(t)=1+\kappa |t|^2 \ (\ell=2).$$
Then by Lemma \ref{lem:4.2} (iii), we can see that
$g(s):= h \big( f(s) \big) f'(s)$ satisfies
$$
\frac{g(s)}{s^{p-1 \over 2}} \to C>0 \ \hbox{as} \ s \to \infty$$
for some $C>0$.
This implies that if $1<p<3$, the nonlinear term of
the converted semilinear problem \eqref{eq:4.3} 
has a sublinear growth at infinity.
\end{remark}

\subsection{Existence of a positive radial solution of \eqref{eq:4.1}}

In this subsection, we prove the existence of a positive radial solution 
$u\in C^2(\mathbb{R}^N)$ of \eqref{eq:4.1}. 
To this end, we show that the function
$$
g(s):= h \big( f(s) \big) f'(s)$$
satisfies (G1), (G3) and (G6) in Section 1.
For $s<0$, we extend $g(s)$ as an odd function.

First we can easily see that $g \in C^1[0,\infty)$.
Moreover since $h(0)=0$ and $f(0)=0$, we also have $g(0)=0$.
Next one has
$$
g'(0)=\lim_{s \to 0} \frac{g(s)}{s}
= \lim_{s \to 0} \frac{ h \big( f(s) \big)}{f(s)}
\frac{f(s)-f(0)}{s} f'(s).$$
Since $f(s) \to 0$ as $s \to 0$, it follows that
$$
\lim_{s \to 0} \frac{ h \big( f(s) \big)}{f(s)}
=\lim_{ t \to 0} \frac{h(t)-h(0)}{t}=h'(0).$$
Thus from (H1), we get
$$
g'(0)=\lim_{s \to 0} \frac{g(s)}{s}=h'(0)f'(0)^2 <0$$
and hence (G1) holds.

To prove (G3), we observe that
\begin{equation*}
G(s)=\int_0^s h \big( f(\tau) \big) f'(\tau)\,d\tau
=H \big( f(s) \big)- H \big( f(0) \big)
=H \big( f(s) \big).
\end{equation*}
Let $\tilde{\zeta}>0$ be a constant in (H3) and put $\zeta=f^{-1}(\tilde{\zeta})$. 
Then from (H3), it follows that $G(\zeta)=H(\tilde{\zeta})>0$
and hence (G3) holds.

Finally we show that (G6) holds. When $N \ge 3$, we have
by (H5) and Lemma \ref{lem:4.2} (iii) that
$$
\limsup_{s \to \infty} \frac{g(s)}{s^{ N+2 \over N-2}}
=\limsup_{s\to \infty}
\frac{ h \big( f(s) \big)}{\big( f(s) \big)^{\frac{(\ell+1)N+2}{N-2}}}
\left( \frac{f(s)}{s^{2 \over \ell+2}} \right)^{ (\ell+2)N \over N-2}
\frac{sf'(s)}{f(s)} \le 0.$$
Next suppose that $N=2$. By Lemma \ref{lem:4.2} (iii), 
there exists $C>0$ such that $f(s)^{\ell +2} \le Cs^2$ for large $s>0$.
Moreover by the definition of $f(s)$ and (A1), 
putting $\displaystyle \underline{a}:=\inf_{t\geq 0}a(t)>0$,
we also have $f'(s) \le \frac{1}{\sqrt{\underline{a}}}$.
Thus from (H5), we obtain
$$
\limsup_{s \to \infty} \frac{g(s)}{e^{\alpha s^2}}
=\limsup_{s \to \infty} 
\frac{ h \big( f(s) \big)}{e^{ \frac{\alpha}{C}f(s)^{\ell +2}}} 
\frac{e^{ \frac{\alpha}{C}f(s)^{\ell +2}}}{e^{\alpha s^2}} f'(s)
\le \frac{1}{\sqrt{\underline{a}}} \limsup_{s\to \infty}
\frac{ h \big( f(s) \big)}{e^{ \frac{\alpha}{C}f(s)^{\ell +2}}}
\le 0.$$

Now since (G1), (G3) and (G6) are satisfied, we can apply results in
\cite{BGK, BL, HIT} to obtain the existence of a
positive radial solution 
$v \in H^1 \cap C^2(\mathbb{R}^N)$ of \eqref{eq:4.3}. 
Then by Lemma \ref{lem:4.3}, 
we obtain the existence of a positive radial solution 
$u \in X \cap C^2(\mathbb{R}^N)$ of \eqref{eq:4.1}.

Finally we note that a positive radial solution of \eqref{eq:4.1} 
obtained here is indeed a ground state of \eqref{eq:4.1}.
To be more precise, we define the energy functional $J: X \to \mathbb{R}$ by
$$J(u):= \frac{1}{2} \int_{\mathbb{R}^N} a(u) |\nabla u|^2 \,dx
-\int_{\mathbb{R}^N} H(u) \,dx,$$
where $X= \left\{ u \in H^1(\mathbb{R}^N) \ ; \ a(u) | \nabla u|^2 \in L^1(\mathbb{R}^N) \right\}$.
Then we can show the existence of a ground state $w$, which is a solution of
\eqref{eq:4.1} satisfying
$$
J(w)= \inf \left\{ J(u) \ ; \ J'(u)=0, \ u\in X \setminus \{0 \} \right\}.$$
Moreover one can also show that $w$ is positive and radially symmetric.

\subsection{Proof of Theorem \ref{thm:4.1}} 

In this subsection, we complete the proof of Theorem \ref{thm:4.1}.
By Theorem \ref{thm:1.1} and Lemma \ref{lem:4.3}, 
it suffices to prove that (G1)-(G5) in Section 1 hold for
\begin{equation}\label{eq:4.4}
 g(s):= h \big( f(s) \big) f'(s).
\end{equation}
Since we have established (G1) and (G3) in Subsection 4.2,
it remains to show that (G2), (G4) and (G5) hold.
To this end, let $\beta$, $\tilde{\beta}>0$ be constants in (H2)
and define $b:=f^{-1}(\beta)$, $\tilde{b}:=f^{-1}(\tilde{\beta})$ respectively.

\begin{lemma}\label{lem:4.5}
The function $g(s)$ defined in \eqref{eq:4.4} satisfies (G2).
\end{lemma}

\begin{proof}
First we observe that $g$ and $h$ have same sign because $f'>0$.
Thus from (H2), it follows that 
$g(s)<0$ for $0<s<b$ and $g(s)>0$ for $b<s<\tilde{b}$.

Next by Lemma \ref{lem:4.2} (ii), we have
\begin{equation}\label{eq:4.5}
g'(s)=h' \big( f(s) \big) f'(s)^2 + h \big( f(s) \big) f''(s)
= \frac{ h' \big( f(s) \big)}{ a \big( f(s) \big)}
-\frac{h \big( f(s) \big) a' \big( f(s) \big)}{2a^2 \big(f(s) \big)}.
\end{equation}
Thus from (A1) and (H2), we obtain
$$
g'(b)=\frac{h'(\beta)}{a(\beta)}-\frac{h(\beta)a'(\beta)}{2a^2(\beta)}
=\frac{h'(\beta)}{a(\beta)}>0.$$
In a similar way, we can show that
$g(s)<0$ on $(\tilde{b},\infty)$ and $g'(\tilde{b})<0$ when $\tilde{b}<\infty$.
This completes the proof.
\end{proof}

In order to prove (G4) and (G5), we prepare the following lemma.

\begin{lemma}
\label{lem:4.6}
$\displaystyle \frac{t \sqrt{a(t)}}{\int_0^t \sqrt{a(\tau)} \,d\tau}$ 
is non-decreasing on $[\beta, \infty)$.
\end{lemma}

\begin{proof}
For simplicity, we put
\begin{equation*}
\phi(t):= \frac{t \sqrt{a(t)}}{\int_0^t \sqrt{a(\tau)} \,d\tau}.
\end{equation*} 
Then from (A1), it follows that $\phi \in C^1[\beta, \infty)$.
Thus it suffices to show that $\phi'(t) \geq 0$ on $[\beta, \infty)$.

Now by a direct computation, one has
\begin{align*}
 \phi'(t)
&=
\left(
\int_0^t \sqrt{a(\tau)} \,d\tau
\right)^{-2}
\left\{
\left(
\sqrt{a(t)} + \frac{t a'(t)}{2 \sqrt{a(t)}}
\right)
\int_0^t \sqrt{a(\tau)} \,d\tau
- t \sqrt{a(t)} \sqrt{a(t)}
\right\} \\
&=
\sqrt{a(t)}
\left(
\int_0^t \sqrt{a(\tau)} \,d\tau
\right)^{-2}
\left\{
\left(
1 + \frac{t a'(t)}{2 a(t)}
\right)
\int_0^t \sqrt{a(\tau)} \,d\tau
- t \sqrt{a(t)}
\right\} \\
&=
\sqrt{a(t)}
\left(
\int_0^t \sqrt{a(\tau)} \,d\tau
\right)^{-2}
\left\{
\frac{t a'(t)}{2 a(t)}
\int_0^t \sqrt{a(\tau)} \,d\tau
- t \sqrt{a(t)}
+\int_0^t \sqrt{a(\tau)} \,d\tau
\right\}. 
\end{align*}
Noticing that $t \ge \beta$, we have from (A3) that
\begin{equation*}
\frac{t a'(t)}{2 a(t)} \geq 
\frac{\tau a'(\tau)}{2 a(\tau)}
\text{ if } \beta \leq \tau \leq t \ \hbox{and} \ 
\frac{t a'(t)}{2 a(t)} 
\geq 
\frac{\beta a'(\beta)}{2 a(\beta)} 
\geq 
\frac{\tau a'(\tau)}{2 a(\tau)}
\text{ if } 0< \tau <\beta
\end{equation*}
Thus we obtain
\begin{align*}
& \frac{t a'(t)}{2 a(t)}
\int_0^t \sqrt{a(\tau)} \,d\tau
=
\int_0^t \frac{t a'(t)}{2 a(t)} \sqrt{a(\tau)} \,d\tau \\
& \geq 
\int_0^t \frac{\tau a'(\tau)}{2 a(\tau)} \sqrt{a(\tau)} \,d\tau 
=
\int_0^t \frac{\tau a'(\tau)}{2 \sqrt{a(\tau)}}  \,d\tau 
=
\int_0^t \tau \left(\sqrt{a(\tau)}\right)' \,d\tau \\
&=
\left[
\tau \left(\sqrt{a(\tau)}\right)
\right]_0^t-
\int_0^t \sqrt{a(\tau)} \,d\tau 
=
t \sqrt{a(t)}
-
\int_0^t \sqrt{a(\tau)} \,d\tau.
\end{align*}
This implies that $\phi' \geq 0$ on $[\beta,\infty)$
and hence the proof is complete.
\end{proof}

\begin{lemma}
\label{lem:4.7}
The function $g(s)$ defined in \eqref{eq:4.4} satisfies (G4) and (G5).
\end{lemma}

\begin{proof}
First by the definition of $K_g(s)$, Lemma \ref{lem:4.2} (ii) and from \eqref{eq:4.5},
it follows that
\begin{equation*}
 K_g(s) = 
\left\{
\frac{h'\big( f(s) \big)}{\sqrt{a \big( f(s) \big)} h \big( f(s) \big)} 
- \frac{a'\big( f(s) \big)}{2a \big( f(s) \big) \sqrt{a \big( f(s) \big)}}
\right\}
\int_0^{f(s)} \sqrt{a(\tau)} \,d\tau.
\end{equation*}
Since $f(b)=\beta$, we have only to show that
\begin{equation*}
\tilde{K}_g(t) :=
\left\{
\frac{h'(t)}{\sqrt{a(t)} h(t)} - \frac{a'(t)}{2a(t) \sqrt{a(t)}}
\right\}
\int_0^{t} \sqrt{a(\tau)} \,d\tau
\end{equation*}
is decreasing on $[\beta,\tilde\beta)$.
Now we can see that $\tilde{K}_g(t)$ is rewritten as
\begin{align*}
\tilde{K}_g(t) =
\frac{
\displaystyle
\frac{t h'(t)}{h(t)}
-
\frac{t a'(t)}{2 a(t)}
}{
\displaystyle
\frac{t \sqrt{a(t)}}{\int_0^t \sqrt{a(\tau)} \,d\tau}
}.
\end{align*}
Thus from (A3), (H4) and by Lemma \ref{lem:4.6}, (G4) is satisfied.

Next to prove (G5), it suffices to show that
$\tilde{K}_g \leq 1$ on $(0,\beta)$.
From (A1), (A2) and (H4), we have
\begin{equation*}
 \tilde{K}_g(t) \leq
\frac{
\int_0^t \sqrt{a(\tau)} \,d\tau
}{
t \sqrt{a(t)}
}
\leq
\frac{
\int_0^t \sqrt{a(t)} \,d\tau
}{
t \sqrt{a(t)}
}
=1 \ \hbox{for} \ t\in (0,\beta).
\end{equation*}
Thus (G5) holds. This completes the proof.
\end{proof}

\subsection{Results for the complex valued ground state for modified Schr\"odinger equation}

In this subsection, we consider a special case $a(t)=1+2\kappa |t|^2$,
$h(t)=|t|^{p-1}t-\lambda t$, and study the uniqueness and the non-degeneracy
of a {\it complex-valued} ground state, 
which are important in the study of the corresponding time-evolution Schr\"odinger equation. 

We consider the following modified Schr\"{o}dinger equation:
\begin{equation*}
i{\partial z \over \partial t}=-\Delta z
- \kappa \Delta (|z|^2)z-|z|^{p-1}z,
    \quad (t,x)\in (0,\infty)\times\mathbb{R}^N, 
\end{equation*}
where $\kappa>0$, $p>1$ and $z: \mathbb{R} \times \mathbb{R}^N \to \mathbb{C}$. 
For physical backgrounds, we refer to \cite{BEPZ, K}.
We are interested in standing waves of the form: 
$z(t,x)=u(x)e^{i\lambda t}$,
where $\lambda>0$ and $u:\mathbb{R}^N \to \mathbb{C}$. 
Then we obtain the following quasilinear elliptic problem:
\begin{equation}
\label{eq:4.6}
-\Delta u+\lambda u-\kappa u \Delta (|u|^2)
=|u|^{p-1}u \quad \hbox{in} \ \mathbb{R}^N.
\end{equation}
From physical as well as mathematical points of view, the most important issue
is the stability of the standing wave.
It is known that in the study of the stability of standing waves,
the uniqueness and the non-degeneracy of ground states of \eqref{eq:4.6}
plays an important role.
(See \cite{CLW, CJS, CO} for results on the (in)stability of standing waves.)

Now we define the energy functional and the energy space by
$$
J(u)=\frac{1}{2} \int_{\mathbb{R}^N} |\nabla u|^2+\kappa |u|^2 \left| \nabla |u| \right|^2 +\lambda |u|^2 \,dx 
-\frac{1}{p+1} \int_{\mathbb{R}^N} |u|^{p+1} \,dx,$$
$$
X_{\mathbb{C}}=\left\{ u\in H^1(\mathbb{R}^N, \mathbb{C}) \ ; \ 
\int_{\mathbb{R}^N} |u|^2 \left| \nabla |u| \right|^2 \,dx< \infty \right\}.
$$
A solution $w$ of \eqref{eq:4.6} is said to be a ground state if it satisfies
$$
J(w)= \inf \left\{ J(u) \ ; \ J'(u)=0, \ u \in X_{\mathbb{C}} \setminus \{ 0\} \right\}.$$
As for the existence and properties of a complex-valued ground state,
we have the following.
For the proof, we refer to \cite{CJS}.
\begin{proposition}\label{prop:4.8}
Suppose $\lambda>0$, $\kappa>0$ and 
$\left\{
\begin{array}{ll}
1<p<\frac{3N+2}{N-2} & \mbox{if} \ N \ge 3, \\
1<p<\infty & \mbox{if} \ N=1,2.
\end{array}
\right.$
Then problem \eqref{eq:4.6} has a ground state $w$, which
has a form $w(x)=e^{i \theta} |w(x)|$ for some $\theta \in \mathbb{R}$.

Moreover let $w$ be a real-valued ground state of \eqref{eq:4.6}. 
Then $w$ satisfies the following properties:
\begin{itemize}
\item[\rm(i)] $w\in C^2(\mathbb{R}^N)$ and $w(x)>0$ for all $x\in \mathbb{R}^N$.
\item[\rm(ii)] $w$ is radially symmetric 
and decreasing with respect to $r=|x|$.
\item[\rm(iii)] There exist $c,\ c'>0$ such that
$$
\lim_{ |x| \rightarrow \infty} e^{\sqrt{\lambda}|x|}(|x|+1)^{{N-1 \over 2}} w(x) =c, \ 
\lim_{r \rightarrow \infty} e^{\sqrt{\lambda}r}(r+1)^{{N-1 \over 2}}
 {\partial w \over \partial r} =-c'.
$$
\end{itemize}
\end{proposition}

Proposition \ref{prop:4.8} tells us that up to a phase shift, 
we may assume that the ground state of \eqref{eq:4.6} is real-valued.
Moreover any ground states are positive and radially symmetric.
Thus the uniqueness of the ground state
of \eqref{eq:4.6} follows from Theorem \ref{thm:4.1}.
Finally it is known that $p= \frac{3N+2}{N-2}$ is the critical exponent
for the existence of non-trivial solutions of \eqref{eq:4.6}.
This can be proved by using the Pohozaev type identity. 
(See \cite{AW1} for the proof.)

Now let $\mathcal{G}$ be the set of ground states of \eqref{eq:4.6}. 
Since \eqref{eq:4.6} is invariant under the translation and the phase shift,
we have the following result.

\begin{theorem} \label{thm:4.9}
Suppose that $\lambda>0$, $\kappa>0$ and 
$\left\{
\begin{array}{ll}
1<p<\frac{3N+2}{N-2} & \mbox{if} \ N \ge 3, \\
1<p<\infty & \mbox{if} \ N=1,2.
\end{array}
\right.$
Let $w$ be the unique (real-valued) ground state of 
\eqref{eq:4.6}. Then we have
$$
\mathcal{G}=\left\{ e^{i\theta} w(\cdot +y) \ ; \ y\in \mathbb{R}^N, 
\ \theta \in \mathbb{R} \right\}.$$
Moreover we also have
$$\displaystyle \mathop{\rm Ker}\,(\mathcal{L})
= \mathop{\rm span} \left\{ iw(x), \ \frac{\partial w}{\partial x_1}, \ \cdots, \ 
\frac{\partial w}{\partial x_N} \right\}.$$
Here $\mathcal{L}$ is the linearized operator of \eqref{eq:4.6} 
around the unique (real-valued) ground state $w$, which is given by
\begin{align*}
\mathcal{L}(\phi)&= -\Delta \phi +\lambda \phi -\kappa
(2w \Delta w+2|\nabla w|^2) \phi \\
&\quad -\kappa w^2 \Delta ( \phi +\overline{\phi})
-2\kappa w \nabla w \cdot \nabla (\phi +\overline{\phi})
-\kappa w \Delta w(\phi+\overline{\phi})
\\
&\quad -w^{p-1} \phi
-\frac{p-1}{2}w^{p-1}(\phi+\overline{\phi}), \ 
\phi\in H^2(\mathbb{R}^N,\mathbb{C}),
\end{align*}
where $\overline{\phi}$ is a complex conjugate of  $\phi$.
\end{theorem}

\begin{proof}
To prove Theorem \ref{thm:4.9}, we put $\phi=\phi_1+i \phi_2$ with
$\phi_1, \phi_2\in H^2(\mathbb{R}^N,\mathbb{R})$ and
decompose $\mathcal{L}$ into two operators $\mathcal{L}_1$, $\mathcal{L}_2$ acting on
$\phi_1$ and $\phi_2$ respectively. By a direct computation, we have
\begin{align}\label{eq:4.7}
\mathcal{L}_1(\phi_1)&= -\Delta \phi_1+\lambda \phi_1-\kappa
(2w\Delta w+2|\nabla w|^2)\phi_1 \nonumber \\
&\quad -2\kappa w^2 \Delta \phi_1-4\kappa w \nabla w \cdot \nabla \phi_1
-2\kappa w \Delta w \phi_1
-pw^{p-1}\phi_1,\\
\mathcal{L}_2(\phi_2)&=
-\Delta \phi_2+\lambda \phi_2
-\kappa(2w \Delta w+2|\nabla w|^2)\phi_2-w^{p-1}\phi_2. \nonumber 
\end{align}
By Corollary \ref{cor:2.4}, Lemma \ref{lem:4.3}
and Theorem \ref{thm:4.1}, it follows that
$$
\mathop{\rm Ker}\,(\mathcal{L}_1)
=\mathop{\rm span} \left\{ \frac{\partial w}{\partial x_1}, \cdots, \frac{\partial w}{\partial x_N} \right\}.
$$

Next we show that $\mathop{\rm Ker}\,(\mathcal{L}_2)={\rm span} \{w \}$.
Although the proof can be found in \cite{S}, 
we give a much simpler proof based on Corollary \ref{cor:2.4}.
By the definition of $\mathcal{L}_2$, one has $w\in \mathop{\rm Ker}\,(\mathcal{L}_2)$.
We suppose by contradiction that there exists $\tilde{w} \in H^1(\mathbb{R}^N)$ such that
$\tilde{w} \not\equiv w$ and $\mathcal{L}_2(\tilde{w})=0$.
This implies that $0$ is not a simple eigenvalue of $\mathcal{L}_2$.
Then by Corollary \ref{cor:2.4} (v), 
it follows that the principal eigenvalue $\mu(\mathcal{L}_2)$ is negative
and the corresponding eigenfunction $\psi$ can be chosen to be positive.

Since $\psi$ satisfies
\begin{equation} \label{eq:4.8}
-\Delta \psi +\lambda \psi -\kappa( 2w \Delta w +2|\nabla w|^2) \psi -w^{p-1}\psi
=\mu \psi,
\end{equation}
multiplying \eqref{eq:4.8} by $w$ and integrating over $\mathbb{R}^N$, we get
$$
\int_{\mathbb{R}^N} \nabla \psi \cdot \nabla w +\lambda \psi w
-\kappa(2w^2 \Delta w+2w|\nabla w|^2) \psi -w^p \psi \,dx
=\mu \int_{\mathbb{R}^N} \psi w \,dx.$$
On the other hand, multiplying \eqref{eq:4.6} with $u=w$ by $\psi$, we also have
$$
\int_{\mathbb{R}^N} \nabla \psi \cdot \nabla w+\lambda \psi w
-\kappa( 2w^2 \Delta w+2w|\nabla w|^2) \psi -w^p \psi \,dx =0.$$
Subtracting these equations, we obtain
$$
0= \mu \int_{\mathbb{R}^N} \psi w \,dx.$$
However since $\mu<0$, $\psi>0$ and $w>0$, this is a contradiction.
This implies that $\mathop{\rm Ker}\,(\mathcal{L}_2)={\rm span} \{w\}$.
This completes the proof of Theorem \ref{thm:4.9}.
\end{proof}

\begin{remark}\label{rem:4.10}
In \cite{ASW1} and \cite{S}, the authors required 
a technical assumption $3 \le p$ to obtain the non-degeneracy of
the ground state of \eqref{eq:4.6}. 
We could remove this additional assumption in Theorem \ref{thm:4.9}. 
The key is to obtain Proposition \ref{prop:2.3} even when
the nonlinear term may have a sublinear growth.

We also note that our result covers the case $N=2$.
In \cite{AW1}, the uniqueness for the case $N=2$ has been obtained
under the assumption $3 \le p$ and some largeness condition on 
$\lambda$ and $\kappa$. Theorem \ref{thm:4.9} enables us to 
obtain the uniqueness without any restrictions on 
$\kappa$, $\lambda$ and $p$.

Finally by Corollary \ref{cor:2.4}, we can obtain
more precise information on the linearized operator
around the unique real-valued ground state.
Indeed let $\mathcal{L}_1$ be the 
linearized operator defined in \eqref{eq:4.7}.
Then we have the following results, which answer 
the question raised in \cite[Remark 5.6]{AW1} and complete previous results 
on the non-degeneracy in \cite{ASW1, S}.
\begin{itemize}
\item[\rm(i)] $\sigma(\mathcal{L}_1)=\sigma_p(\mathcal{L}_1) \cup \sigma_e(\mathcal{L}_1)$, 
$\sigma_e(\mathcal{L}_1)=[\lambda, \infty)$ and 
$\sigma_p(\mathcal{L}_1) \subset (-\infty, \lambda)$.
\item[\rm(ii)] If $\mu \in \sigma_p(\mathcal{L}_1)$, 
then the corresponding eigenfunction $\varphi(x)$ satisfies
$$
|\varphi(x)| \le C_{\epsilon} e^{-\sqrt{{ \lambda-\mu+\varepsilon \over 2}}|x|},   \ x\in \mathbb{R}^N
$$
for any small $\varepsilon>0$ and some $C_{\epsilon}>0$.
\item[\rm(iii)] If $\mu \in \sigma_p(\mathcal{L}_1) \cap (-\infty,0)$, 
then the corresponding eigenfunction is radially symmetric.
\item[\rm(iv)] The principal eigenvalue $\mu_1(\mathcal{L}_1)<0$ is simple,
 and the corresponding eigenfunction
$\varphi_1$ can be chosen to be positive.
\item[\rm(v)] The second eigenvalue $\mu_2(\mathcal{L}_1)$ is zero
and 
$ \mathop{\rm Ker}\,(\mathcal{L}_1)= \mathrm{span} \, 
\left\{ {\partial w \over \partial x_i} \right\}_{i=1}^N.
$
\end{itemize}
\end{remark}

\appendix 
\section{Sturm Comparison Principle}
In this appendix, we introduce a version of the Sturm Comparison
Principle which was used in Section 3.

\begin{lemma}[{\cite[Lemma 5, p. 497]{Mc}}]\label{lem:A.1}

Let $U$ and $V$ be continuous functions and satisfy
\begin{align*}
U''+\frac{N-1}{r} U'+g(r)U=0,\\
V''+\frac{N-1}{r} V'+G(r)V=0
\end{align*}
respectively on some interval $(\mu, \nu) \subset [0,\infty)$.
Suppose that $g$ and $G$ are continuous, 
$G \ge g$ and $G \not\equiv g$ on $(\mu, \nu)$.
Assume further that one of the following conditions holds.
\begin{itemize}
\item[\rm(a)] $\mu>0$ and $U(\mu)=U(\nu)=0$.
\item[\rm(b)] $\mu=0$ and $U'(\mu)=V'(\mu)=U(\nu)=0$.
\end{itemize}

Then $V$ has at least one zero on $(\mu,\nu)$.
\end{lemma}

\subsection*{Acknowledgment}
This paper was carried out while the second author was staying at
University Bordeaux I. The author is very grateful to all the staff of 
University Bordeaux I for their kind hospitality.
The second author is supported by 
JSPS Grant-in-Aid for Scientific Research (C) (No. 15K04970).


\begin{thebibliography}{99}

\bibitem{AW1} S.~Adachi, T.~Watanabe,
\emph{Uniqueness of the ground state solutions of quasilinear Schr\"odinger equations},
Nonlinear Anal. {\bf 75} (2012), 819--833.

\bibitem{AW2} S.~Adachi, T.~Watanabe,
\emph{Asymptotic uniqueness of ground states for a class of
quasilinear Schr\"odinger equations with $H^1$-supercritical exponent}, 
J. Diff. Eqns. {\bf 260} (2016), 3086--3118.

\bibitem{ASW1} S.~Adachi, M.~Shibata, T.~Watanabe,
\emph{Global uniqueness results for ground states for a class of
quasilinear elliptic equations}, preprint.

\bibitem{BS} P.~Bates, J.~Shi,
\emph{Existence and instability of spike layer solutions to singular perturbation problems},
J. Funct. Anal. {\bf196} (2002), 429--482.

\bibitem{BGK} H.~Berestycki, T.~Gallou\"et, O.~Kavian,
\emph{ Equations de champs scalaires euclidens non
lin\'eaires dans le plan},
C. R. Acad. Paris Ser. I Math. {\bf 297} (1984), 307--310.

\bibitem{BL} 
H.~Berestycki and P.~L.~Lions,
\emph{Nonlinear scalar fields equations, I. Existence of a ground state},
Arch. Rational Mech. Anal. {\bf 82} (1983), 313--345.

\bibitem{BEPZ} L.~Brizhik, A.~Eremko, B.~Piette, W.~J.~Zahkrzewski,
\emph{Static solutions of a $D$-dimensional modified nonlinear Schr\"odinger equation},
Nonlinearity {\bf 16} (2003), 1481--1497.

\bibitem{BJM} J.~Byeon, L.~Jeanjean, M.~Maris,
\emph{Symmetry and monotonicity of least energy solutions},
Calc. Var. PDE {\bf 36} (2009), 481--492.

\bibitem{CLW} J.~Chen, Y.~Li, Z.~Q.~Wang,
\emph{Stability of standing waves for a class of quasilinear Schr\"odinger equations},
European J. Appl. Math. {\bf 23} (2012), 611--633.

\bibitem{C} C.~V.~Coffman,
\emph{Uniqueness of the ground state solution for $\Delta u-u+u^3=0$ and a variational characterization of other solutions},
Arch. Rational Mech. Anal. {\bf 46} (1972), 81--95.

\bibitem{CJS} M.~Colin, L.~Jeanjean, M.~Squassina,
\emph{Stability and instability results for standing waves of
quasi-linear Schr\"{o}dinger equations}, 
Nonlinearity. {\bf 23} (2010), 1353--1385.

\bibitem{CO} M.~Colin, M.~Ohta,
\emph{Instability of ground states for a quasilinear Schr\"odinger equation},
Diff. Int. Eqns. {\bf 27} (2014), 613--624.

\bibitem{CEF} C.~Cort\'azar, M.~Elgueta, P.~Felmer,
\emph{Uniqueness of positive solutions of $\Delta u+f(u)=0$ in $\mathbb{R}^N$, $N\geq 3$},
Arch. Rational Mech. Anal. {\bf 142} (1998), 127--141.

\bibitem{FW} A.~Floer and A.~Weinstein, 
\emph{Nonspreading wave packets for the cubic Schr\"{o}dinger equation with a bounded potential}, J. Funct. Anal. {\bf 69} (1986), 397--408.

\bibitem{GNN} B.~Gidas, W.~M.~Ni, L.~Nirenberg,
\emph{Symmetry of positive solutions of nonlinear elliptic equations in $\mathbb{R}^N$}
Adv. Math. Studies A {\bf7} (1981), 369--402.

\bibitem{GS} F.~Gladiali, M.~Squassina,
\emph{Uniqueness of ground states for a class of quasi-linear elliptic equations},
Adv. Nonlinear Anal. {\bf 1} (2012), 159--179.

\bibitem{HIT} 
J.~Hirata, N.~Ikoma and K.~Tanaka,
\emph{Nonlinear scalar field equations in $\mathbb{R}^N$: mountain pass and symmetric mountain pass approaches},
Top. Methods in Nonlinear Anal. {\bf 35} (2010), 253--276.


\bibitem{Ko} P.~Korman,
\emph{A global approach to ground state solutions},
Elect. J. Diff. Eqns. {\bf 122} (2008), 1--13.

\bibitem{K} S.~Kurihara,
\emph{Large-amplitude quasi-solitons in superfluid films},
J. Phys. Soc. Japan {\bf50} (1981), 3262--3267.

\bibitem{Kw} M.~K.~Kwong,
\emph{Uniqueness of positive solutions of $\Delta u-u+u^p=0$ in $\mathbb{R}^N$},
Arch. Rat. Mech. Anal. {\bf 105} (1989), 243--266.

\bibitem{M} M.~Mari\c{s},
\emph{Existence of nonstationary bubbles in higher dimensions},
J. Math. Pure Appl. {\bf 81} (2002), 1207--1239.

\bibitem{Mc} K.~Mcleod,
\emph{Uniqueness of positive radial solutions of
$\Delta u+f(u)=0$ in $\mathbb{R}^N$, {\rm II}}, 
Trans. AMS. {\bf 339} (1993), 495--505.

\bibitem{MS} K.~Mcleod, J.~Serrin,
\emph{Uniqueness of positive radial solutions of $\Delta u=f(u)=0$ in $\mathbb{R}^N$},
Arch. Rational Mech. Anal. {\bf 99} (1987), 115--145.

\bibitem{NT} W.~M.~Ni, I.~Takagi,
\emph{Locating the peaks of least-energy solutions to a semilinear Neumann problem},
Duke Math. J. {\bf 70} (1993), 247--281.

\bibitem{OS} T.~Ouyang, J.~Shi,
\emph{Exact multiplicity of positive solutions for a class of semilinear
problems: II},
J. Diff. Eqns. {\bf 158} (1999), 94--151.

\bibitem{PS1} L.~A.~Peletier, J.~Serrin,
\emph{Uniqueness of positive solutions of semilinear equations in $\mathbb{R}^n$}, 
Arch. Rational Mech. Anal. {\bf 81} (1983), 181--197.

\bibitem{PS2} L.~A.~Peletier, J.~Serrin,
\emph{Uniqueness of non-negative solutions of semilinear equations in $\mathbb{R}^n$},
J. Diff. Eqns. {\bf 61} (1986), 380--397.

\bibitem{S} A.~Selvitella,
\emph{Nondegeneracy of the ground state for quasilinear Schr\"odinger equations}, Calc. Var. PDE {\bf 53} (2015), 349--364.

\bibitem{ST} J.~Serrin, M.~Tang,
\emph{Uniqueness of ground states for quasilinear elliptic equations},
Indiana Univ. Math. J. {\bf 49} (2000), 897--923.

\bibitem{St} C.~Stuart,
\emph{ Lectures on the orbital stability of standing waves and applications to
the nonlinear Schr\"odinger equation},
Milan J. Math. {\bf 76} (2008), 329--399.

\end{thebibliography}
\end{document}